\documentclass[11pt,pdftex]{article}
\pdfoutput=1
\usepackage{url}
\usepackage{graphics}
\usepackage{graphicx}
\usepackage{color}
\usepackage{amsmath}
\usepackage{amsfonts}
\usepackage{amsthm,amssymb}
\usepackage{tikz,color}
\usepackage[abbrev]{amsrefs}
\usepackage{mdwlist}
\usepackage{lineno}

\def\nmodk#1#2{{#1\,\text{\rm mod}\,#2}}

\newcommand{\Cr}{{\hbox{\rm cr}}}
\newcommand{\ucr}{{\hbox{\rm cr}}}

\newcommand\floor[1]{{\lfloor{#1}\rfloor}}
\newcommand\ceil[1]{{\lceil{#1}\rceil}}

\newcommand\biggfloor[1]{{\biggl\lfloor{#1}\biggr\rfloor}}
\newcommand{\bigfloor}[1]{{\bigl\lfloor{#1}\bigr\rfloor}}

\def\Auxone{{\floor{ (k+1)^2/4 }}}

\def\auxone{{\bigfloor{\frac{(k+1)^2}{4}}}}

\newcommand{\beq}{\begin{equation}}
\newcommand{\eeq}{\end{equation}}
\newcommand{\beann}{\begin{eqnarray*}}
\newcommand{\eeann}{\end{eqnarray*}}
\newcommand{\bc}{\begin{center}}
\newcommand{\ec}{\end{center}}

\newtheorem{theorem}{Theorem}
\newtheorem{observation}[theorem]{Observation}
\newtheorem{lemma}[theorem]{Lemma}

\newtheorem{corollary}[theorem]{Corollary}
\newtheorem{proposition}[theorem]{Proposition}

\title{Book drawings of complete bipartite graphs}

\author{Etienne de Klerk\thanks{School of Physical \& Mathematical Sciences, Nanyang Technological University,
             Singapore, 
and Department of Econometrics and Operations Research, Tilburg University, The Netherlands.}
  \and
Dmitrii V.~Pasechnik\thanks{School of Physical \& Mathematical Sciences, Nanyang Technological University,
             Singapore.}
             \and
             Gelasio Salazar\thanks{Instituto de Fisica,
Universidad Autonoma de San Luis Potosi,
San Luis Potosi, SLP Mexico 78000. Supported by CONACYT Grant 106432.
}}

\begin{document}

\linenumbers

\maketitle

\begin{abstract}
We recall that a {\em book with $k$ pages} consists of a straight line
(the {\em spine}) and $k$ half-planes (the {\em pages}), such that the
boundary of each page is the spine. If a graph is drawn on a book with
$k$ pages in such a way that the vertices lie on the spine, and each
edge is contained in a page, the result is a {\em k-page book drawing}
(or simply a {\em $k$-page drawing}). The {\em pagenumber} of a graph
$G$ is the minimum $k$ such that $G$ admits a $k$-page embedding (that
is, a $k$-page drawing with no edge crossings).  The $k$-{\em page
  crossing number} $\nu_k(G)$ of $G$ is the minimum number of
crossings in a $k$-page drawing of $G$.  We investigate the
pagenumbers and $k$-page crossing numbers of complete bipartite
graphs. We find the exact pagenumbers of several complete bipartite
graphs, and use these pagenumbers to find the exact $k$-page crossing number
of $K_{k+1,n}$ for $k\in \{3,4,5,6\}$. We also prove the general
asymptotic estimate $\lim_{k\to\infty} \lim_{n\to\infty}
\nu_k(K_{k+1,n})/(2n^2/k^2)=1$. Finally, we give general upper bounds
for $\nu_k(K_{m,n})$, and relate these bounds to the $k$-planar
crossing numbers of $K_{m,n}$ and $K_n$.
\end{abstract}


{\bf Keywords:} $2$-page crossing number, book crossing number,
complete bipartite graphs, Zarankiewicz conjecture

{\bf AMS Subject Classification:} 90C22, 90C25, 05C10, 05C62, 57M15, 68R10

\section{Introduction}

In~\cite{leighton}, Chung, Leighton, and Rosenberg proposed the model
of embedding graphs in books. We recall that a
{\em book} consists of a line (the {\em spine}) and $k\ge 1$
half-planes (the {\em pages}), such that the boundary of each page is
the spine. In a {\em book embedding}, each edge is drawn on a single
page, and no edge crossings are allowed.
The {\em pagenumber} (or {\em book thickness}) $p(G)$ of a graph $G$ is the minimum $k$ such that $G$
can be embedded in a $k$-page
book~\cites{Bernhart-Kainen,dw,leighton,Kainen}. Not surprisingly,
determining the pagenumber of an arbitrary graph is
NP-Complete~\cite{leighton}.

In a {\em book drawing} (or
{\em $k$-page drawing}, if the book has $k$ pages), each
edge is drawn on a single page, but edge crossings are allowed.
The {\em $k$-page crossing number
  $\nu_k(G)$} of a graph $G$ is the minimum number of crossings in a
$k$-page drawing of $G$.

Instead of using a straight line as the spine and halfplanes as pages,
it is sometimes convenient to visualize a $k$-page drawing using the
equivalent {\em circular model}.  In this model, we view a $k$-page drawing of a
graph $G=(V,E)$ as a set of $k$ circular drawings of graphs $G^{(i)} =
(V, E^{(i)})$ $(i=1,\ldots,k)$, where the edge sets $E^{(i)}$ form a
$k$-partition of $E$, and such that the vertices of $G$ are arranged
identically in the $k$ circular drawings.  In other words, we assign each edge in $E$ to
exactly one of the $k$ circular drawings. In Figure~\ref{fig:k453pages}
we illustrate a $3$-page drawing of $K_{4,5}$ with $1$ crossing.

\begin{figure}[h!]
\begin{center}
\includegraphics[width=15cm]{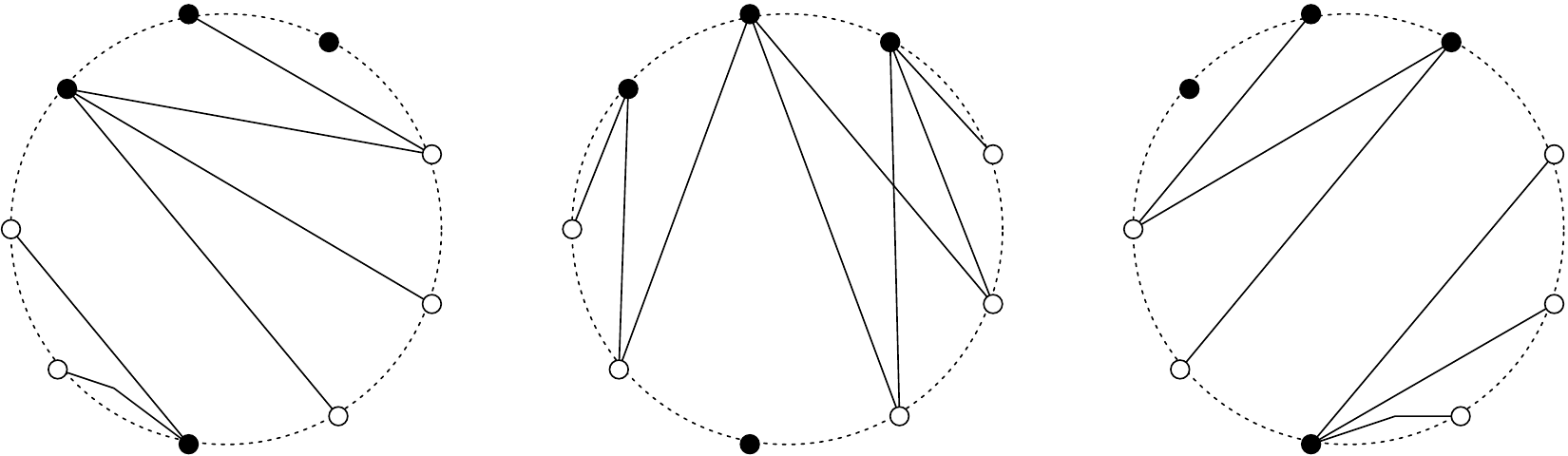}
\caption{\label{fig:k453pages} \small A $3$-page drawing of $K_{4,5}$
  with $1$ crossing. Vertices in the chromatic class of size $4$ are
  black, and vertices in the chromatic class of size $5$ are white. We
  have proved (Theorem~\ref{thm:exact1}) that the pagenumber of
  $K_{4,5}$ is $4$, and so it follows that $\nu_3(K_{4,5}) \ge 1$. Now
  this $1$-crossing drawing implies that $\nu_3(K_{4,5}) \le 1$, and
  so it follows that $\nu_3(K_{4,5}) = 1$.}
\end{center}
\end{figure}

Very little is known about the pagenumbers or $k$-page crossing
numbers of interesting families of graphs.
Even computing the pagenumber of planar graphs
is a nontrivial task; Yannakakis proved~\cite{Yannanakis} that
four pages always suffice, and sometimes are required, to embed a
planar graph.
It is a standard exercise to show that
the pagenumber $p(K_n)$ of the complete graph $K_n$ is $\ceil{n/2}$. Much less is known about the $k$-page crossing
numbers of complete graphs. A thorough treatment of $k$-page crossing
numbers (including estimates for $\nu_k(K_n)$), with
general lower and upper bounds, was offered by Shahrokhi et
al.~\cite{sssv96}.
In~\cite{dps}, de Klerk et al.~recently
used a variety of techniques to compute several exact $k$-page
crossing numbers of complete graphs, as well as to give some
asymptotic estimates.


Bernhart and Kainen~\cite{Bernhart-Kainen} were the first to
investigate the pagenumbers of complete bipartite graphs, giving lower
and upper bounds for $p(K_{m,n})$. The upper bounds
in~\cite{Bernhart-Kainen} were then improved by Muder, Weaver, and
West~\cite{muder}. These upper bounds were further improved by
Enomoto, Nakamigawa, and Ota~\cite{Enomoto}, who derived the best
estimates known to date.
Much less is known about the $k$-page crossing number of
$K_{m,n}$.


\subsection{$1$-page drawings of $K_{m,n}$}

Although calculating the $1$-page crossing number of the complete graph $K_n$ is
trivial, this is by no means the case for
the complete bipartite graph $K_{m,n}$.
Still, our knowledge about $\nu_1(K_{m,n})$ is almost completely
satisfactory, due to the following result by
Riskin~\cite{riskin}.

\begin{theorem}[Riskin~\cite{riskin}] If $m\bigl|n$ then $\nu_1(K_{m,n}) =
  \frac{1}{12} n(m -1)(2mn - 3m - n)$, and this
minimum value is attained when the $m$ vertices are distributed evenly amongst
the $n$ vertices.
\end{theorem}



\subsection{$2$-page drawings of $K_{m,n}$}

Zarankiewicz's Conjecture states that the (usual)
crossing number $\ucr(K_{m,n})$ of $K_{m,n}$ equals
$Z(m,n):=\floor{\frac{m}{2}}\floor{\frac{m-1}{2}}\floor{\frac{n}{2}}\floor{\frac{n-1}{2}}$,
for all positive integers $m,n$. Zarankie\-wicz~\cite{zaran} found
drawings of $K_{m,n}$ with exactly $Z(m,n)$ crossings, thus proving
$\ucr(K_{m,n}) \le Z(m,n)$. These drawings can be easily adapted to
$2$-page drawings (without increasing the number of crossings), and so
it follows that $\nu_2(K_{m,n}) \le Z(m,n)$.

Since $\ucr(G) \le \nu_2(G)$ for any  $G$, Zarankiewicz's
Conjecture implies the (in principle, weaker) conjecture
$\nu_2(K_{m,n}) = Z(m,n)$. Zarankiewicz's
Conjecture has been verified (for $\ucr(K_{m,n})$, and thus also
for $\nu_2(K_{m,n})$) for $\min\{m,n \}\le 6$~\cite{kleitman}, and for
the special cases $(m,n)\in \{(7,7),(7,8), (7,9), (7,10), (8,8),
(8,9),$ $ (8,10)\}$~\cite{woodall}. Recently, de Klerk and
Pasechnik~\cite{dp} used semidefinite programming techniques to prove
that $\lim_{n\to\infty} \nu_2(K_{7,n})/Z(7,n) = 1$.


\subsection{$k$-page drawings of $K_{m,n}$ for $k \ge 3$: lower bounds}\label{sub:low}

As far as we know, neither exact results nor estimates for
$\nu_k(K_{m,n})$ have been reported in the literature, for any $k \ge 3$.
Indeed, all the nontrivial results known about $\nu_k(K_{m,n})$ are
those that can be indirectly derived from the thorough investigation
of Shahrokhi, S\'ykora, Sz\'ekely, and Vrt'o on multiplanar crossing
numbers~\cite{sssv07}.

We recall that a {multiplanar drawing} is similar to a book
drawing, but involves unrestricted planar drawings. Formally, let
$G=(V,E)$ be a graph. A $k$-{\em planar drawing} of $G$ is a set of $k$
planar drawings of graphs $G^{(i)} = (V, E^{(i)})$ $(i=1,\ldots,k)$, where the edge sets
$E^{(i)}$ form a $k$-partition of $E$. Thus, to obtain the
$k$-planar drawing, we take the
drawings of the graphs $G^{(i)}$, and (topologically) identify the $k$
copies of each vertex. The $k$-{\em planar crossing number}
$\Cr_k(G)$ of $G$ is the minimum number of crossings in a $k$-planar
drawing of $G$.  A {\em multiplanar drawing} is a $k$-planar drawing
for some positive integer $k$.

It is very easy to see that $\nu_k(G) \ge \Cr_{\ceil{k/2}}(G)$, for every
graph $G$ and every nonnegative integer $k$. Thus lower bounds of
multiplanar (more specifically,
$r$-planar) crossing numbers immediately imply lower bounds of book
(more specifically, $2r$-page) crossing numbers.  A strong result by Shahrokhi, S\'ykora,
Sz\'ekely, and Vrt'o is the exact determination of the $r$-planar
crossing number of $K_{2r+1,n}$ (\cite{sssv07}*{Theorem 3}:
\begin{equation}\nonumber
\ucr_r(K_{2r+1,n})=
\biggfloor{
\frac{n}{2r(2r-1)}
}
\biggl(
n-r(2r-1)
\biggl(
\biggfloor{
\frac{n}{2r(2r-1)}
}
-1
\biggr)
\biggr).
\end{equation}

Using this result and our previous observation $\nu_k(G) \ge
\Cr_{\ceil{k/2}}(G)$, one obtains:

\begin{theorem}[Follows from~\cite{sssv07}*{Theorem 3}]\label{thm:fromsssv}
For every even integer $k$ and every integer $n$,
\[
\text{
\hglue 2.3 cm
}
\nu_k(K_{k+1,n}) \ge
\biggfloor{
\frac{n}{k(k-1)}
}
\biggl(
n-{{\frac{k}{2}}}\biggl(k-1\biggr)
\biggl(
\biggfloor{
\frac{n}{k(k-1)}
}
-1
\biggr)
\biggr).
\text{
\hglue 2.4 cm $\Box$
}
\]
\end{theorem}


Regarding general lower bounds,
using the following inequality from~\cite{sssv07}*{Theorem 5}
\begin{equation}\nonumber 
\ucr_r(K_{m,n}) \ge \frac{1}{3(3r-1)^2}{m\choose 2}{n\choose 2},
\text{\rm \hglue 0.4 cm for } \, m\ge
6r-1 \text{\rm \, and\, } n\ge \max\{6r-1,2r^2\},
\end{equation}
and the observation $\nu_k(K_{m,n}) \ge
\ucr_{\ceil{k/2}}(K_{m,n})$, one obtains
\begin{equation}\label{eq:lobokp}
\nu_k(K_{m,n}) \ge \frac{1}{3(3\ceil{\frac{k}{2}}-1)^2}{m\choose 2}{n\choose
  2},
\text{\rm for } \, m\ge
6\ceil{k/2}-1 \text{\rm \, and\, } n\ge \max\{6\ceil{k/2}-1,2\ceil{k/2}^2\}.
\end{equation}

We finally remark that slightly better bounds can be obtained in the
case $k=4$, using the bounds for biplanar crossing numbers by
Czabarka, S\'ykora, Sz\'ekely, and Vrt'o~\cites{bip1,bip2}.


\subsection{$k$-page drawings of $K_{m,n}$ for $k \ge 3$: upper bounds}

We found no references involving upper bounds of $\nu_k(K_{m,n})$
in the literature. We note that since not every $\ceil{k/2}$-planar
drawing can be adapted to a $k$-page drawing, upper bounds for
$\ceil{k/2}$-planar crossing numbers do not yield upper bounds for
$k$-page crossing numbers, and so the results on $(k/2)$-planar drawings
of $K_{m,n}$ in~\cite{sssv07} cannot be used to derive upper bounds for
$\nu_k(K_{m,n})$.

Below (cf.~Theorem~\ref{thm:upp1}) we shall give general upper
bounds for $\nu_k(K_{m,n})$. We derive these bounds using a natural
construction, described in Section~\ref{sec:cons}.


\section{Main results}\label{sec:mainr}

In this section we state the main new results in this paper, and briefly discuss the
strategies of their proofs.

\subsection{Exact pagenumbers}

We have calculated the exact pagenumbers of several complete bipartite graphs:

\begin{theorem}~\label{thm:exact1}
For each $k\in \{2,3,4,5,6\}$, the pagenumber of
$K_{k+1,\Auxone+1}$ is $k+1$.
\end{theorem}

The proof of this statement is computer-aided, and is based on the
formulation of $\nu_k(K_{m,n})$ as a vertex coloring problem on an
associated graph.
This is presented in Section~\ref{sec:pe1}.

By the clever construction by Enomoto, Nakamigawa, and Ota~\cite{Enomoto},
$K_{k+1,\Auxone}$ can be embedded into $k$ pages,
and so Theorem~\ref{thm:exact1} implies (for $k\in\{2,3,4,5,6\}$) that
$\floor{\frac{(k+1)^2}{4}}+1$ is the smallest value of $n$ such that
$K_{k+1,n}$ does {\em not} embed in $k$ pages. The case  $k=2$ follows
immediately
from the nonplanarity of $K_{3,3}$; we have included this value in the
statement for completeness.

\subsection{The $k$-page crossing number of $K_{k+1},n$: exact results
and bounds}

Independently of the intrinsic value of learning some exact
pagenumbers, the importance of Theorem~\ref{thm:exact1} is that we
need these results in order to establish the following general result.
We emphasize that we follow the convention that $\binom{a}{b}=0$ whenever $a<b$.


\begin{theorem}\label{thm:main1}
Let $k\in \{2,3,4,5,6\}$, and let $n$ be any positive integer.
Define $\ell:= \auxone$ and
$q:=\nmodk{n}{\bigfloor{\frac{(k+1)^2}{4}}}$.
Then
\[
\nu_k(K_{k+1,n}) =
q\cdot \binom{\frac{n-q}{\ell}+1}{2}
+
\bigl(
\ell-q
\bigr)
\cdot \binom{\frac{n-q}{\ell}}{2}
.
\]
\end{theorem}

In this statement we have included the case $k=2$ again for completeness, as it
asserts the
well-known result that the $2$-page crossing number of $K_{3,n}$ equals $Z(3,n)=\bigfloor{\frac{n}{2}}
\bigfloor{\frac{n-1}{2}}$.

Although our techniques do not yield the exact value of
$\nu_k(K_{k+1,n})$ for other values of $k$, they give lower and
upper bounds that imply sharp asymptotic estimates:

\begin{theorem}\label{thm:main2}
Let $k, n$ be positive integers. Then
\[
2n^2
\biggl(
\frac{1}{k^2 + 2000k^{7/4}}
\biggr)
-n
<  \nu_k(K_{k+1,n}) \le
\frac{2n^2}{k^2} + \frac{n}{2}.
\]
Thus
\[
\lim_{k\to\infty} \left( \lim_{n\to\infty} \frac{\nu_k(K_{k+1,n})}{2n^2/k^2}\right) = 1.
\]
\end{theorem}

To grasp how this result relates to
the bound in Theorem~\ref{thm:fromsssv}, let us note that the
corresponding estimate (lower bound) from Theorem~\ref{thm:fromsssv} is
$\lim_{k\to\infty}\bigl(\lim_{n\to\infty} \nu_k(K_{k+1,n})/(2n^2/k^2)\bigr) \ge
1/4$. Theorem~\ref{thm:main2} gives the exact asymptotic value of this quotient.

In a nutshell, the strategy to prove the lower bounds in Theorems~\ref{thm:main1}
and~\ref{thm:main2} is to establish lower bounds for $\nu_k(K_{k+1,n})$
obtained under the assumption
that $\nu_k(K_{k+1,s+1})$ cannot be $k$-page embedded (for some
integer $s:=s(k)$). These results put the burden of the proof of the lower
bounds in Theorems~\ref{thm:main1} and~\ref{thm:main2} in finding good
estimates of $s(k)$. For $k=3,4,5,6$ (Theorem~\ref{thm:main1}), these come from
Theorem~\ref{thm:exact1}, whereas for $k > 6$
(Theorem~\ref{thm:main2}) these are obtained
from~\cite{Enomoto}*{Theorem 5}, which gives general estimates for such
integers $s(k)$. The lower bounds for $\nu_k(K_{k+1,n})$ needed for
both Theorems~\ref{thm:main1} and~\ref{thm:main2} are established in
Section~\ref{sec:lower}.

In Section~\ref{sec:upper} we prove the upper bounds on
$\nu_k(K_{k+1,n})$ claimed in Theorems~\ref{thm:main1}
and~\ref{thm:main2}. To obtain these bounds, first we find a
particular kind of $k$-page embeddings of $K_{k+1,\Auxone}$, which we
call {\em balanced} embeddings.  These embeddings are inspired by,
although not equal to, the embeddings described by Enomoto et
al.~in~\cite{Enomoto}. We finally use these embeddings to construct
drawings of $\nu_k(K_{k+1,n})$ with the required number of crossings.

Using the lower and upper bounds derived in Sections~\ref{sec:lower}
and~\ref{sec:upper}, respectively,
Theorems~\ref{thm:main1} and~\ref{thm:main2} follow easily;
their proofs are given in Section~\ref{sec:proofs}.

\subsection{General upper bounds for $\nu_k(K_{m,n})$}

As we mentioned above, we found no general upper
bounds for $\nu_k(K_{m,n})$ in the literature. We came across a rather
natural
way of drawing $K_{m,n}$ in $k$ pages, that yields the general upper
bound given in the following statement.


\begin{theorem}\label{thm:upp1}
Let $k, m,n$ be nonnegative integers. Let $r:=m$ mod $k$ and $s:=n$
mod $k$. Then
\[
\nu_k(K_{m,n})\le 
\frac{(m-r)(n-s)}{4k^2}(m-k+r)(n-k+s)
\le  \frac{1}{k^2} {m\choose 2}{n\choose 2}.
\]
\end{theorem}

The proof of this statement is given in Section~\ref{sec:cons}.

\subsection{$k$-page vs.~$(k/2)$-planar crossing numbers}

As we have already observed, for every even integer $k$, every
$k$-page drawing can be regarded as a $(k/2)$-planar drawing. Thus,
for every graph $G$, $\ucr_{k/2}(G) \le \nu_k(G)$.

Since there is (at least in principle) considerable more freedom in a
$(k/2)$-planar drawing than in a $k$-page drawing, it is natural to
ask whether or not this additional freedom can be translated into
a substantial saving in the number of crossings. For small values of
$m$ or $n$, the answer is yes. Indeed, Beineke~\cite{beineke} described how to
draw $K_{k+1,k(k-1)}$ in $k/2$ planes without crossings, but
by Proposition~\ref{pro:extw},
$K_{k+1,k^2/4+500k^{7/4}}$ cannot be $k$-page embedded; thus  the
$k/2$-planar crossing number of $K_{k+1,k(k-1)}$ is $0$, whereas its
$k$-page crossing number can be arbitrarily large. Thus it makes sense
to ask about the asymptotic behaviour when $k,m$, and $n$ all go to infinity.
Letting $\gamma(k):=
\lim_{m,n\to\infty} \ucr_{k/2}(K_{m,n})/\nu_k(K_{m,n})$, we focus on
the question: is $\lim_{k\to\infty} \gamma(k) = 1$?

Since we do not know (even asymptotically) the $(k/2)$-planar or the $k$-page
crossing number of $K_{m,n}$, we can only investigate this question
in the light of the current best bounds available.

In Section~\ref{sec:conrem} we present a discussion around this
question. We conclude that if the $(k/2)$-planar and the $k$-page
crossing numbers (asymptotically) agree with the current best upper
bounds, then indeed the limit above equals $1$.
We also observe that this is not the case for complete graphs: the
currently best known $(k/2)$-planar drawings of $K_n$ are substantially better
(even asymptotically) than the currently best known $k$-page drawings
of $K_n$.



\section{Exact pagenumbers: proof of Theorem~\ref{thm:exact1}}\label{sec:pe1}

We start by observing that for every integer $n$, the graph $K_{k+1,n}$
can be embedded in $k+1$ pages, and so the pagenumber
$p(K_{k+1,\Auxone+1})$ of
$K_{k+1,\Auxone+1}$ is at most $k+1$.  Thus we need to show the reverse
inequality
$p(K_{k+1,\Auxone+1}) \ge k+1$, for every $k\in \{3,4,5,6\}$.
It clearly suffices to show that
$\nu_k(K_{k+1,\Auxone+1}) >0 $, for every $k\in \{3,4,5,6\}$.

These inequalities are equivalent to $k$-colorability of certain auxiliary graphs. To this end, we define an auxiliary graph
$G_{D}(K_{m,n})$ associated with a $1$-page (circular) drawing $D$ of $K_{m,n}$ as follows. The vertices
of $G_{D}(K_{m,n})$ are the edges of $K_{m,n}$, and two vertices are adjacent if  the corresponding edges cross in the drawing $D$.

We immediately have the following result, that is essentially due to Buchheim and Zheng~\cite{Buccheim-Zheng}.

\begin{lemma}[cf.\ Buchheim-Zheng \cite{Buccheim-Zheng}]
\label{lemma:chromatic}
One has
$\nu_k(K_{m,n}) >0 $ if and only if the chromatic number of $G_{D}(K_{m,n})$ is greater than $k$ for all circular drawings $D$ of $K_{m,n}$.
\end{lemma}

As a consequence we may decide if $\nu_k(K_{m,n}) >0 $ by considering all possible circular drawings $D$ of $K_{m,n}$, and computing the chromatic
numbers of the associated auxiliary graphs $G_{D}(K_{m,n})$.
The number of distinct circular drawings of $K_{m,n}$ may be computed using the classical {\em orbit counting lemma},  often attributed to
Burnside, although it was certainly already known to Frobenius.

\begin{lemma}[Orbit counting lemma]
Let a finite group $ \mathcal{G}$ act on a finite set $ \Omega.$ Denote by
$ \Omega^g,$ for $ g\in \mathcal{G},$ the set of elements of $ \Omega$ fixed by $ g.$
Then the number $ N$ of orbits of $ \mathcal{G}$ on $ \Omega$ is the average, over $ \mathcal{G},$
of $ |\Omega^g|,$ i.e.
$$ N=\frac{1}{|\mathcal{G}|}\sum_{g\in \mathcal{G}} |\Omega^g|.$$
\end{lemma}

We will apply this lemma by considering that a circular drawing of $K_{m,n}$ is uniquely determined by
the ordering of the $m$ blue and $n$ red vertices on a circle.
We therefore define the finite set $\Omega$ as the set of all ${m+n \choose n}$ such orderings.
Now consider the usual action of the dihedral group $\mathcal{G} := D_{m+n}$ on the set $\Omega$.
For our purposes two orderings are the same, i.e.\ correspond to the same circular drawing of $K_{m,n}$, if they belong to the same orbit of $\mathcal{G}$.
We therefore only need to count the number of orbits by using the last lemma. The final result is as follows. (We omit the details of the counting argument, as it is a straightforward exercise in
combinatorics.)

\begin{lemma}
\label{lemma:number of circular drawings of Kmn}
Let $m$ and $n$ be positive integers and denote $d = \mbox{gcd}(m,n)$.
The number  of distinct circular drawings of $K_{m,n}$ equals:
\[
\frac{1}{2(m+n)}\left\{
\begin{array}{ll}
 \frac{m+n}{2}\left(\dbinom{\frac{m+n}{2}}{n/2} + \dbinom{\frac{m+n-2}{2}}{m/2} + \dbinom{\frac{m+n-2}{2}}{n/2}\right) + \displaystyle\sum_{k=0}^{d-1} \dbinom{\frac{m+n}{o(k)}}{\frac{m}{o(k)}}
& \mbox{\rm ($m$, $n$ even),} \\

 (m+n)\dbinom{\frac{m+n-1}{2}}{n/2} + \displaystyle\sum_{k=0}^{d-1} \dbinom{\frac{m+n}{o(k)}}{\frac{m}{o(k)}}
& \mbox{\rm ($m$ odd, $n$ even),} \\

 (m+n)\dbinom{\frac{m+n-2}{2}}{(m-1)/2} + \displaystyle\sum_{k=0}^{d-1} \dbinom{\frac{m+n}{o(k)}}{\frac{m}{o(k)}}
& \mbox{\rm ($m$, $n$ odd),} \\
\end{array}
\right.
\]
where $o(k)$ is the minimal number between 1 and $d$ such that $k\cdot
o(k)\equiv 0\mod d$.
In other words, $o(k)$ is the order of the subgroup generated by $k$ in the additive group of integers $\mod d$.
\end{lemma}

In what follows we will present computer-assisted proofs that the chromatic number of $G_{D}(K_{m,n})$ is greater than $k$, for specific integers $k,m,n$.
We do not need to compute the chromatic number exactly if we can prove that it is lower bounded by a value strictly greater than $k$.
A suitable lower bound for our purposes is the Lov\'asz $\vartheta$-number.

\begin{lemma}[Lov\'asz \cite{Lovasztheta}]
Given a graph $G = (V,E)$ and the value
\[
\vartheta(G) := \max_{X \succeq 0} \left. \left\{\sum_{i,j \in V} X_{ij} \; \right| \; X_{ij} = 0 \mbox{ if } (i,j) \in E, \; \mbox{trace}(X) = 1, \; X \in \mathbb{R}^{V \times V}\right\},
\]
one has
\[
\omega(\bar G) \le \vartheta(G) \le \chi(\bar G),
\]
where $\omega(\bar G)$ and $\chi(\bar G)$ are the clique and chromatic numbers of the complement $\bar G$ of $G$, respectively.
\end{lemma}
The  $\vartheta(G)$-number may be computed for a given graph $G$ by using semidefinite programming software.
For our computation we used the software DSDP \cite{dsdp}.
\begin{corollary}
\label{cor:theta}
If, for given positive integers $m,n$ and $k$, $\vartheta(\overline{G_{D}(K_{m,n})}) > k$ for all circular drawings $D$ of $K_{m,n}$, then
$\nu_k(K_{m,n}) >0 $.
\end{corollary}

If, for a given circular drawing $D$, we find that
$\vartheta(\overline{G_{D}(K_{m,n})}) = k$, then we  compute the
chromatic number of $G_{D}(K_{m,n})$ exactly,
by using satisfiability or integer programming software.
For our computation we used the satisfiability solver Akmaxsat \cite{akmaxsat}, and for the integer programming formulation the solver XPRESS-MP \cite{xpressmp}.
The formulation of the chromatic number as the solution of a maximum satisfiability problem is described in \cite[\S 3.3]{MR}.
The integer programming formulation we used is the following.

For given $G = (V,E)$ with adjacency matrix $A$, and set of colors $C = \{1,\ldots,k\}$, define the binary variables
\[
x_{ij} = \left\{
\begin{array}{ll}
1 & \mbox{if vertex $i$ is assigned color $j$}, \\
0 & \mbox{else,} \\
\end{array}
\right. \quad (i \in V, \; j \in C),
\]
and consider the integer programming feasibility problem:
\begin{equation}
\label{eq:mip formulation}
\mbox{ Find an } x \in \{0,1\}^{V \times C} \mbox{ such that }  \sum_{j \in C} x_{ij} = 1 \; \forall  i \in V, \; \sum_{i \in V} A_{pi}x_{ij} \le |E|(1-x_{pj}) \; \forall p \in V, j \in C.
\end{equation}

\begin{lemma}
\label{lemma:ip}
A given graph $G = (V,E)$ is $k$-colorable if and only if the integer program \eqref{eq:mip formulation} has a solution.
\end{lemma}
We may therefore solve \eqref{eq:mip formulation} with $G =
G_{D}(K_{m,n})$, for each circular drawing $D$ of $K_{m,n}$, to decide if $\nu_k(K_{m,n}) >0 $.

Finally we describe the results we obtained by using the computational framework described in this section.

\subsubsection*{Case $k=3$: proof of $\nu_3(K_{4,5}) >0$.}
By Lemma \ref{lemma:number of circular drawings of Kmn}, there are $10$ distinct circular drawings $D$ of $K_{4,5}$.
For each $D$ we showed numerically that
$\vartheta(\overline{G_{D}(K_{4,5})}) > 3$.
The required result now follows from Corollary \ref{cor:theta}.

\subsubsection*{Case $k=4$: proof of $\nu_4(K_{5,7}) >0$.}
By Lemma \ref{lemma:number of circular drawings of Kmn}, there are $38$ distinct circular drawings $D$ of $K_{5,7}$.
For all but one $D$ we showed numerically that
$\vartheta(\overline{G_{D}(K_{5,7})}) > 4$. The remaining case was settled by showing
$\chi({G_{D}(K_{5,7})}) > 4$ using the satisfiability reformulation.
The required result now follows from Corollary \ref{cor:theta} and Lemma \ref{lemma:chromatic}.

\subsubsection*{Case $k=5$: proof of $\nu_5(K_{6,10}) >0$.}
By Lemma \ref{lemma:number of circular drawings of Kmn}, there are $210$ distinct circular drawings $D$ of $K_{6,10}$.
For all but one $D$ we showed numerically that
$\vartheta(\overline{G_{D}(K_{6,10})}) > 5$. The remaining case was settled by showing
$\chi({G_{D}(K_{6,10})}) > 5$ using the satisfiability reformulation.
The required result now follows from Corollary \ref{cor:theta} and Lemma \ref{lemma:chromatic}.

\subsubsection*{Case $k=6$: proof of $\nu_6(K_{7,13}) >0$.}
By Lemma \ref{lemma:number of circular drawings of Kmn}, there are $1980$ distinct circular drawings $D$ of $K_{7,13}$.
For all but one $D$ we showed numerically that
$\vartheta(\overline{G_{D}(K_{7,13})}) > 6$. The remaining case was settled by showing
$\chi({G_{D}(K_{7,13})}) > 4$ using the integer programming reformulation \eqref{eq:mip formulation}.
The required result now follows from Corollary \ref{cor:theta}, Lemma \ref{lemma:ip}, and Lemma \ref{lemma:chromatic}.

\section{$k$-page crossing numbers of $K_{k+1,n}$: lower bounds}\label{sec:lower}

Our aim in this section is to establish lower bounds for
$\nu_k(K_{k+1,n})$. Our strategy is as follows. First we find
(Proposition~\ref{pro:low1}) a
lower bound under the assumption that $K_{k+1,s+1}$ cannot be $k$-page
embedded (for some integer $s:=s(k)$).
Then we find values of $s$ such that $K_{k+1,s+1}$ cannot be
$k$-page embedded; these are given in Propositions~\ref{pro:exon} (for
$k\in\{2,3,4,5,6\}$) and~\ref{pro:extw} (for every $k$). We then put
these results together and establish the lower bounds required in
Theorem~\ref{thm:main1} (see Lemma~\ref{lem:exon}) and in
Theorem~\ref{thm:main2} (see Lemma~\ref{lem:exth}).

\begin{proposition}\label{pro:low1}
Suppose that $K_{k+1,s+1}$ cannot be $k$-page embedded. Let $n$ be a
positive integer, and define $q:=n$ mod $s$.
Then
\[
\nu_k(K_{k+1,n}) \ge
q\cdot \binom{\frac{n-q}{s}+1}{2} + (s-q)\cdot \binom{\frac{n-q}{s}}{2}.
\]
\end{proposition}

\begin{proof}
It is readily verified that if $n\le s$ then the right hand side of
the inequality in the proposition equals $0$,  and so in this case the
inequality trivially holds. Thus we may assume that $n \ge s+1$.

Let $D$ be a $k$-page drawing of $K_{k+1,n}$. Construct an auxiliary
graph $G$ as follows. Let $V(G)$ be the set of
$n$ degree-$(k+1)$ vertices in $K_{k+1,n}$, and join two vertices
$u,v$ in $G$
by an edge if there are edges $e_u, e_v$ incident with $u$ and $v$
(respectively) that cross each other in $D$.

Since $K_{k+1,s+1}$ cannot be embedded in $k$ pages, it follows that
$G$ has no independent set of size $s+1$. Equivalently, the complement
graph $\overline{G}$ of $G$ has no clique of size $s+1$. Tur\'an's theorem
asserts that $\overline{G}$ cannot have more edges than the Tur\'an graph
$T(n,s)$, and so $G$ has at least as many edges as the complement
$\overline{T}(n,s)$ of $T(n,s)$. We recall that $\overline{T}(n,s)$
is formed by the
disjoint union of $s$ cliques, $q$ of them with $(n-q)/s +1$ vertices,
and $s-q$ of them with $(n-q)/s$ vertices. Thus
\[
|E(G)| \ge
q\cdot \binom{\frac{n-q}{s}+1}{2} + (s-q)\cdot \binom{\frac{n-q}{s}}{2}.
\]

Since clearly the number of crossings in $D$ is at least $|E(G)|$, and $D$ is an arbitrary $k$-page drawing of $K_{k+1,n}$, the result follows.
\end{proof}

\begin{proposition}\label{pro:exon}
For each $k\in\{2,3,4,5,6\}$, $K_{k+1,\auxone+1}$ cannot be $k$-page embedded.
\end{proposition}

\begin{proof}
This is an immediate consequence of Theorem~\ref{thm:exact1}.
\end{proof}

\begin{proposition}\label{pro:extw}
For each positive integer $k$, the graph $K_{k+1,k^2/4 + 500 k^{7/4}}$ cannot be $k$-page embedded.
\end{proposition}

\begin{proof}
Define $g(n):=\min\{m\,|\,\text{\rm the pagenumber of $K_{m,n}$ is
  $n$}\}$. Enomoto et al.~proved that $g(n) = n^2/4 + O(n^{7/4})$
(\cite[Theorem 5]{Enomoto}). Our aim is simply to get an explicit
estimate of the $O(n^{7/4})$ term (without making any substantial effort to
optimize the coefficient of $n^{7/4}$).

In the proof of~\cite[Theorem 5]{Enomoto}, Enomoto et al.~gave upper
bounds for three quanitites $m_1, m_2, m_3$, and proved that $g(n) \le
m_1 + m_2 + m_3$.  They showed $m_1 \le n^{3/4}(n-r)$ (for certain
$r\le n$), $m_2 \le (n^{1/4}+1)(2n^{1/4} + 2)(n-1)$, and
$m_3 \le (n^{1/4}+1)(n^{1/4}+2)(2n^{1/4}+3)(n-1) + r'(n-r')$ (for
  certain $r'\le r$).

Noting that $r' \le r$, the inequality $m_1 \le n^{3/4} (n-r)$ gives
$m_1 \le n^{3/4}(n-r')$.
Elementary manipulations give $m_2 <  2n(n^{1/4}+1)^2 = 2n(n^{1/2}+2n^{1/4} + 1) \le
2n(4n^{1/2})=8n^{3/2} < 8n^{7/4}$, and $m_3 < n(2n^{1/4} + 3)^3 +
r'(n-r') < n(5n^{1/4})^3 + r'(n-r') = 125n^{7/4} + r'(n-r')$. Thus we
obtain
$g(n) \le m_1 + m_2 + m_3 < 133n^{7/4} + (n^{3/4} + r')(n-r')$. An
elementary calculus argument shows that $(n^{3/4} + r')(n-r')$ is
maximized when $r'=(n-n^{3/4})/2$, in which case
$(n^{3/4} + r')(n-r') = (1/4)(n^{3/4} + n)^2 = (1/4)(n^{3/2} +
2n^{7/4} + n^2) < (1/4)(n^2 + 3n^{7/4})$.

Thus we get
$g(n) < 133n^{7/4} + n^2/4 + (3/4)n^{7/4} < n^2/4 + 134n^{7/4}$, and
so
$g(k+1) < (k+1)^2/4 + 134(k+1)^{7/4} < k^2/4 + k/2 + 1/4+ 134(2k)^{7/4} <
k^2/4 + 136(2k)^{7/4} = k^2/4 + 136\cdot 2^{7/4} \cdot k^{7/4} < k^2/4
+ 500 k^{7/4}$.
This means, from the definition of $g$, that
$K_{k+1, k^2/4 + 500 k^{7/4}}$ cannot be $k$-page embedded.
\end{proof}

\begin{lemma}\label{lem:exon}
For each $k\in\{2,3,4,5,6\}$, and every integer $n$,
\[
\nu_k(K_{k+1,n})
\ge
q \cdot \binom{\frac{n-q}{\ell}+1}{2}
+
(\ell-q)\cdot \binom{\frac{n-q}{\ell}}{2}
,
\]
where $\ell:=\floor{(k+1)^2/4}$ and  $q:= \nmodk{n}{\floor{(k+1)^2/4}}$.
\end{lemma}

\begin{proof}
It follows immediately from Propositions~\ref{pro:low1} and~\ref{pro:exon}.
\end{proof}

\begin{lemma}\label{lem:exth}
For all positive integers $k$ and $n$,
\[
\nu_k(K_{k+1,n})
>
2n^2
\biggl(
\frac{1}{k^2 + 2000k^{7/4}}
\biggr)
-n.
\]
\end{lemma}

\begin{proof}
By Proposition~\ref{pro:extw} it follows that $K_{k+1,k^2/4 +
  500k^{7/4}}$ cannot be $k$-page embedded. Thus, if we let $s:=k^2/4
+ 500k^{7/4} - 1$ and $q:=\nmodk{n}{s}$, it follows from
Proposition~\ref{pro:low1} that
$
\nu_k(K_{k+1,n}) \ge
q\cdot \binom{\frac{n-q}{s}+1}{2} + (s-q)\cdot \binom{\frac{n-q}{s}}{2}
\ge s\cdot \binom{\frac{n-q}{s}+1}{2} > (n-q)^2/2s$. Thus we have
\begin{equation*}
\nu_k(K_{k+1,n})
>  \frac{(n-q)^2}{2s} > \frac{(n-s)^2}{2s} >  \frac{n^2}{2s}
- n
> 2n^2\biggl(
\frac{1}{k^2 + 2000k^{7/4}}
\biggr)
-n.\qedhere
\end{equation*}
\end{proof}


\section{$k$-page crossing numbers of $K_{k+1,n}$: upper bounds}\label{sec:upper}

In this section we derive an upper bound for the $k$-page
crossing number of $K_{k+1,n}$, which will yield the upper bounds
claimed in both Theorems~\ref{thm:main1} and~\ref{thm:main2}.

To obtain this bound we proceed as
follows. First, we show in Proposition~\ref{pro:upp1} that if for some
$s$ the graph $K_{k+1,s}$ admits a certain kind of
$k$-page embedding (what we call a {\em balanced} embedding), then this
embedding can be used to construct drawings of $\nu_k(K_{k+1,n})$ with
a certain number of crossings. Then we prove, in
Proposition~\ref{pro:ek}, that $K_{k+1,\Auxone}$ admits a balanced
$k$-page embedding for every $k$.  These results are then put together
to obtain the required upper bound, given in Lemma~\ref{lem:corone}.

\subsection{Extending balanced $k$-page embeddings to $k$-page drawings}

We consider $k$-page embeddings of $K_{k+1,s}$, for some integers $k$
and $s$. To help
comprehension, color the
$k+1$ degree-$s$ vertices {\em black}, and the $s$ degree-$(k+1)$
vertices {\em white}.
Given such an embedding, a white vertex $v$, and a page, the {\em load}
of $v$ in this page is the number of edges incident with $v$
that lie on the given page.

The pigeon-hole principle shows that in an $k$-page embedding of
$K_{k+1,s}$, for each white vertex $v$ there must exist a page with
load at least $2$. A $k$-page embedding of $K_{k+1,s}$ is {\em
  balanced} if for each white vertex $v$, there exist $k-1$ pages in
which the load of $v$ is $1$ (and so the load of $v$ in the other page
is necessarily $2$).

\begin{proposition}\label{pro:upp1}
Suppose that $K_{k+1,s}$ admits a balanced $k$-page embedding. Let $n
\ge s$, and define $q:=n$ mod $s$.  Then
\[
\nu_k(K_{k+1,n}) \le
q \cdot \binom{\frac{n-q}{s}+1}{2}
+
(s-q)\cdot \binom{\frac{n-q}{s}}{2}.
\]
\end{proposition}

\begin{proof}
Let $\Psi$ be a balanced $k$-page embedding of $K_{k+1,s}$, presented
in the circular model.
To
construct from $\Psi$ a $k$-page drawing of $K_{k+1,n}$, we first
``blow up'' each white point as follows.

Let $t\ge 1$ be an integer. Consider a white point $r$ in the circle,
and let $N_r$ be a small neighborhood of $r$, such that no point (black or
white) other than $r$ is in $N_r$. Now place $t-1$ additional white points
on the circle, all contained in $N_r$, and let each new white point be
joined to a black point $b$ (in a given page) if and only if $r$ is
joined to $b$ in that page. We say that the white point $r$ has been
{\em converted into a $t$-cluster}.

To construct a $k$-page drawing of $K_{k+1,n}$, we start by choosing
(any) $q$ white points, and then convert each of these $q$ white points
into an $((n-q)/s +1)$-cluster. Finally, convert each of
the remaining $s-q$ white points into an $((n-q)/s)$-cluster. The result
is evidently an $k$-page drawing $D$ of $K_{k+1,n}$.

We finally count the number of crossings in $D$. Consider the
$t$-cluster $C_r$ obtained from some white point $r$ (thus, $t$ is either
$(n-q)/s$ or $(n-q)/s+1$), and consider any page $\pi_i$. It is clear that
if the load of $r$ in $\pi_i$ is $1$, then no edge incident with a
vertex in $C_r$ is crossed in $\pi_i$. On the other hand, if the load
of $r$ in $\pi_i$ is $2$, then it is immediately checked that the
number of crossings involving edges incident with vertices in $C_r$ is
exactly $\binom{t}{2}$. Now the load of $r$ is $2$ in exactly one
page (since $\Psi$ is balanced), and so it follows that the total
number of crossings in $D$ involving edges incident with vertices in
$C_r$ is $\binom{t}{2}$. Since to obtain $D$, $q$ white points were
converted into $((n-q)/s +1)$-clusters, and
$s-q$ white points were converted into $((n-q)/s)$-clusters, it follows
that the number of crossings in $D$ is exactly
\[
q\cdot \binom{\frac{n-q}{s}+1}{2} + (s-q)\cdot \binom{\frac{n-q}{s}}{2}.\qedhere
\]
\end{proof}

\subsection{Constructing balanced $k$-page embeddings}

Enomoto, Nakamigawa, and Ota~\cite{Enomoto} gave a clever general construction
to embed $K_{m,n}$ in $s$ pages for (infinitely) many values of $m,n$, and $s$. In
particular, their construction yields $k$-page embeddings of
$K_{k+1,\Auxone}$. However, the embeddings obtained from their technique are not
balanced (see Figure~\ref{fig:notbal}). We have adapted their construction
to establish the following.

\begin{figure}[h!]
\begin{center}
\scalebox{0.36}{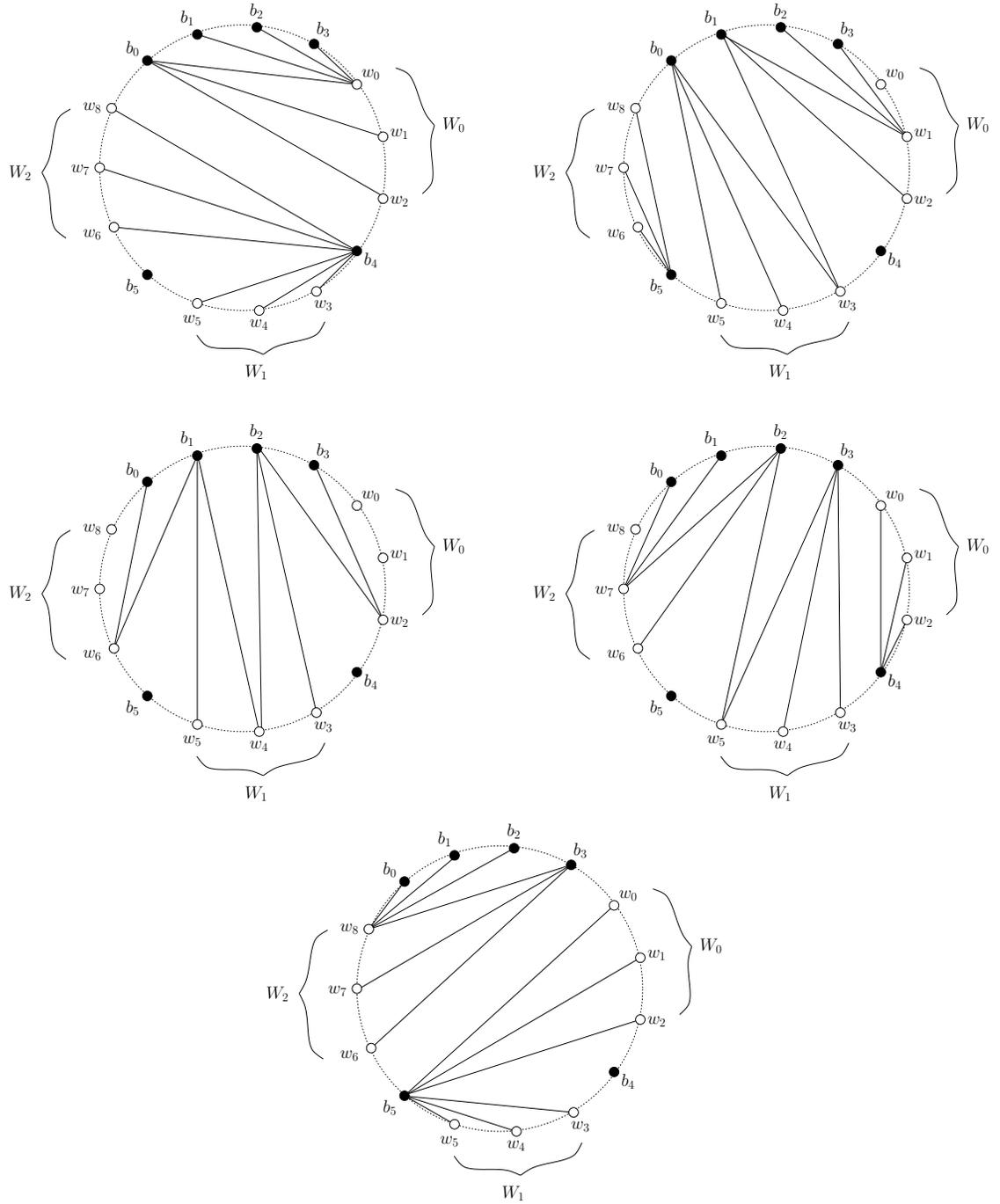}
\end{center}
\caption{\small The $5$-page embedding of $K_{6,9}$
  obtained from the general construction of Enomoto, Nakamigawa, and
  Ota. This embedding is not balanced (for instance, the white
  vertex $w_0$ has degree $4$ in Page $0$).
  }
\label{fig:notbal}
\end{figure}

\begin{proposition}\label{pro:ek}
For each positive integer $k$, the graph $K_{k+1,\floor{(k+1)^2/4}}$
admits a balanced $k$-page embedding.
\end{proposition}

\begin{proof}
We show that for each pair of positive integers $s,t$ such that $t$
is either $s$ or $s+1$, the graph $K_{s+t,st}$
admits a balanced $(s+t-1)$-page embedding.
The proposition then follows: given $k$,
if we set $s:=\floor{(k+1)/2}$ and $t:=\ceil{(k+1)/2}$, then
$t\in\{s,s+1\}$, and clearly $k+1=s+t$ (and so $k=s+t-1$) and $\auxone=st$.

To help comprehension, we color the
$s+t$ degree-$st$ vertices {\em black}, and we color the $st$ degree-$(s+t)$
vertices {\em white}.  We describe the required embedding using the
circular model. Thus, we start with $s+t-1$ pairwise disjoint copies of a
circle; these copies are the pages $0,1,\ldots,s+t-2$. In the boundary
of each copy we place the $s+t + st$ vertices, so that the vertices are
placed in an identical manner in all $s+t-1$ copies. Each edge will be
drawn in the interior of the circle of exactly one page, using the
straight segment joining the corresponding vertices.

\begin{figure}[h!]
\begin{center}
\scalebox{0.5}{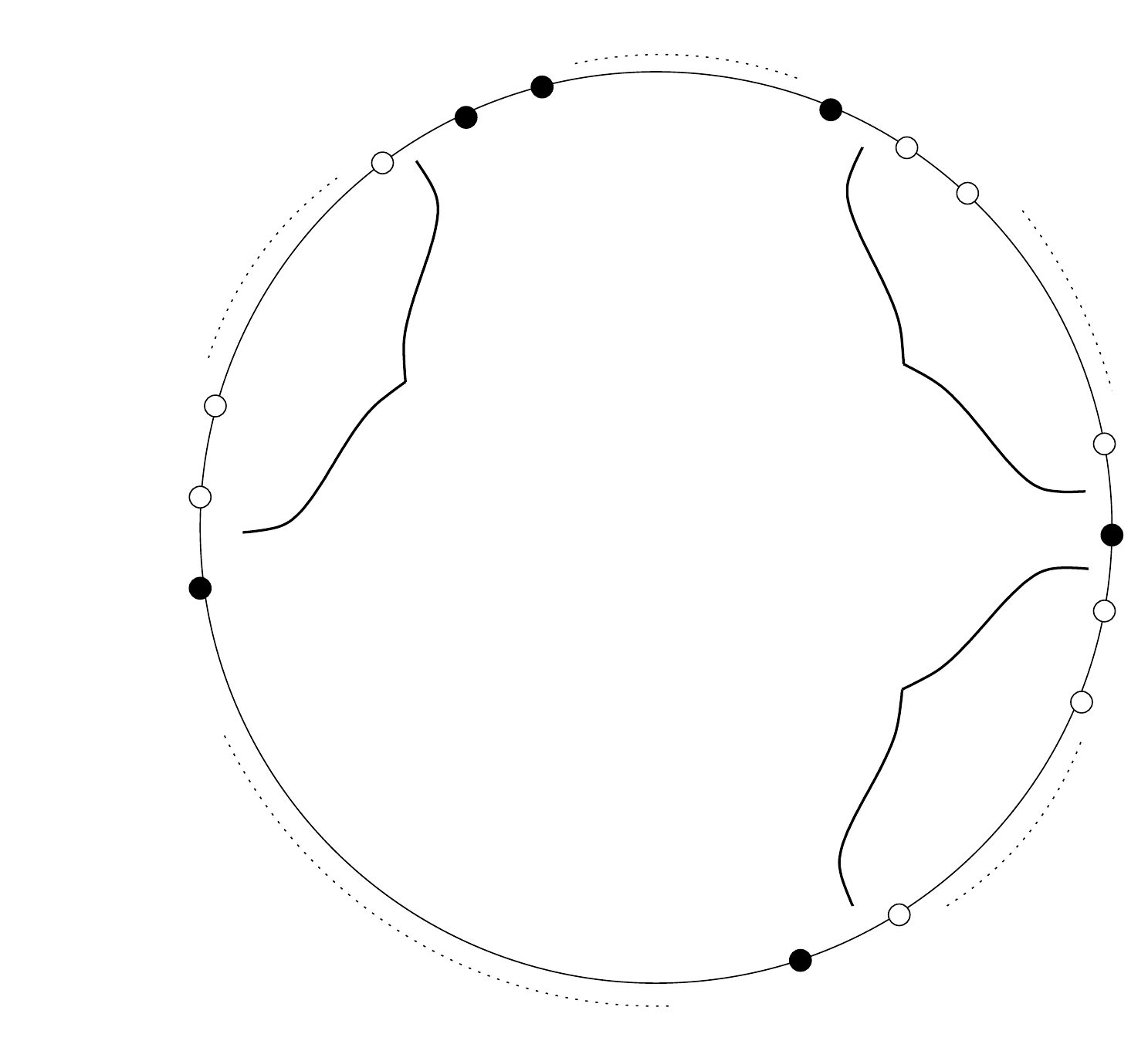}
\caption{
Layout of the vertices of  $K_{s+t,st}$.
}
\label{fig:rough1}
\end{center}
\end{figure}

We now describe how we arrange the white and the black points on the
circle boundary. We use the black-and-white arrangement proposed by Enomoto et
al.~\cite{Enomoto}.
We refer the reader to Figure~\ref{fig:rough1}.
First we place the $s+t$ black points $b_0, b_1,
\ldots, b_{s+t-1}$ in the circle boundary, in this clockwise cyclic order.
Now for each $i\in \{ 0,1,\ldots,t-1\}$, we insert between the
vertices $b_{s+i}$ and $b_{s+i+1}$ a collection $w_{is}, w_{is+1},
\ldots, w_{is+s-1}$ of white vertices, also listed in the clockwise
cyclic order in which they appear between $b_{s+i}$ and $b_{s+i+1}$
(operations on the indices of the black vertices are modulo
$s+t$). For any $i,j$ such that $0 \le i \le j < st$, we let
$W[i:j]$ denote the set of white vertices
$\{w_i,w_{i+1},\ldots,w_j\}$.
For $i=\{0,1,\ldots,t-1\}$,
we call the set $W[is:is+s-1]=\{w_{is},w_{is+1},\ldots,w_{is+s-1}\}$
a {\em white block}, and denote it by $W_i$.
Thus the whole collection of white vertices $w_0, w_1,
\ldots, w_{st-1}$ is partitioned into $t$ blocks $W_0, W_1, \ldots,
W_{t-1}$, each of size $s$.
Note that the black vertices $b_0, b_1, \ldots, b_s$ occur
consecutively in the circle boundary (that is, no white vertex is
between $b_i$ and $b_{i+1}$, for $i\in \{0,1,\ldots,s-1\}$). On the
other hand, for $i=s+1, s+2, \ldots, s+t-1$, the black vertex $b_i$ occurs
between two white vertices: loosely speaking, $b_i$ is sandwiched
between the white blocks $W_{i-1}$ and $W_i$ (operations on the
indices of the white blocks are modulo $t$).

\begin{figure}[h!]
\scalebox{0.36}{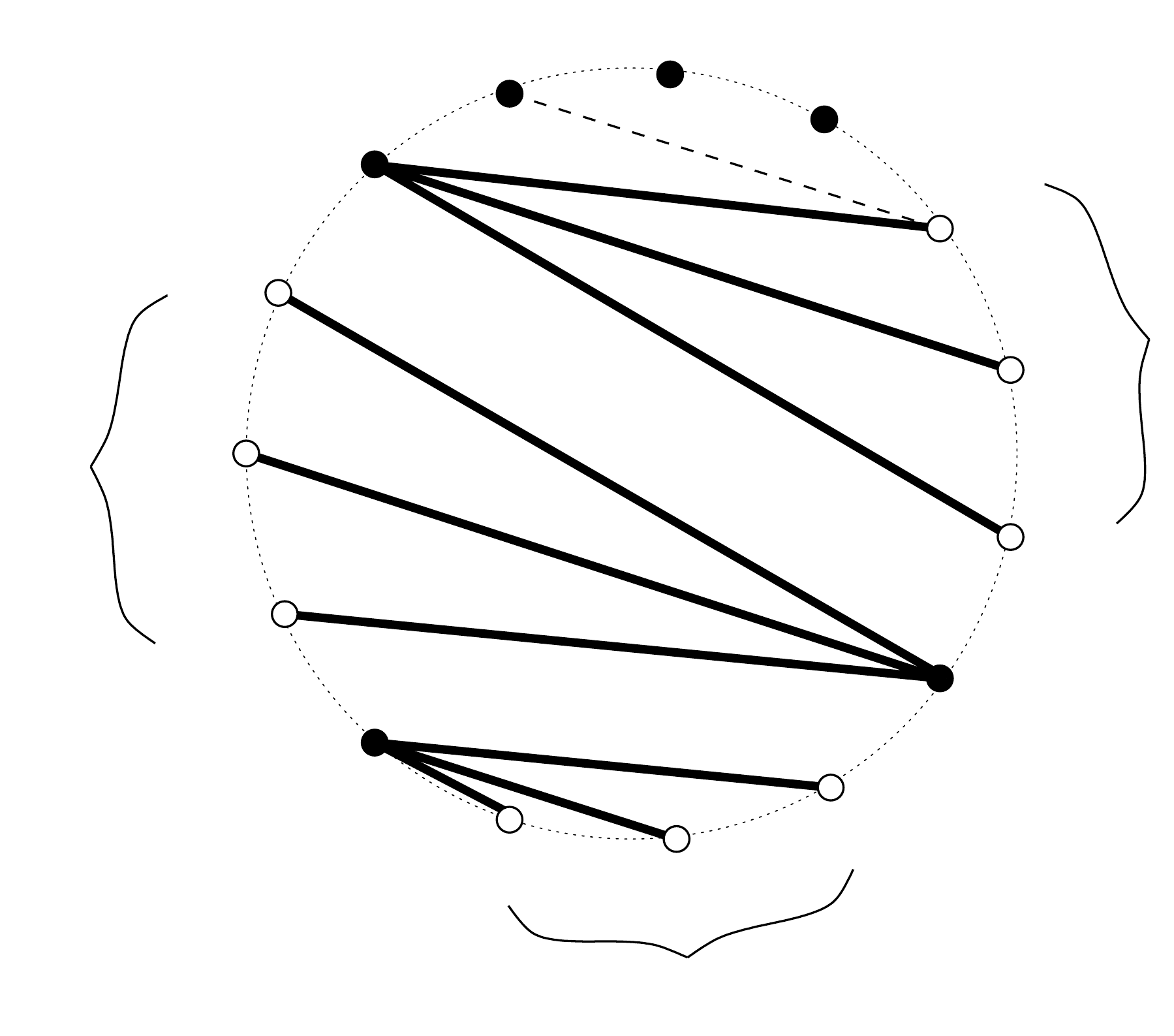}
\hglue 0.4 cm\scalebox{0.36}{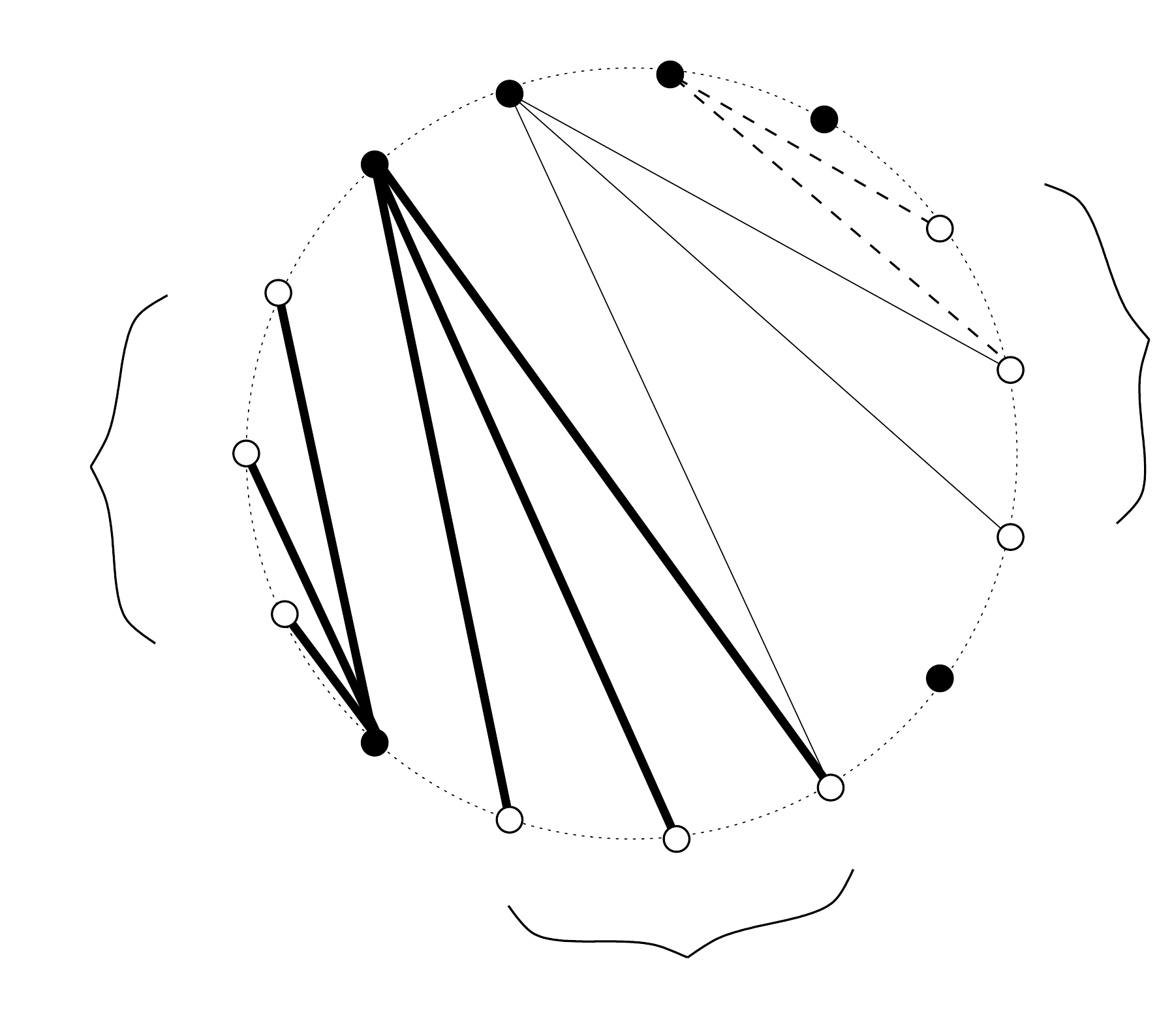}

\scalebox{0.36}{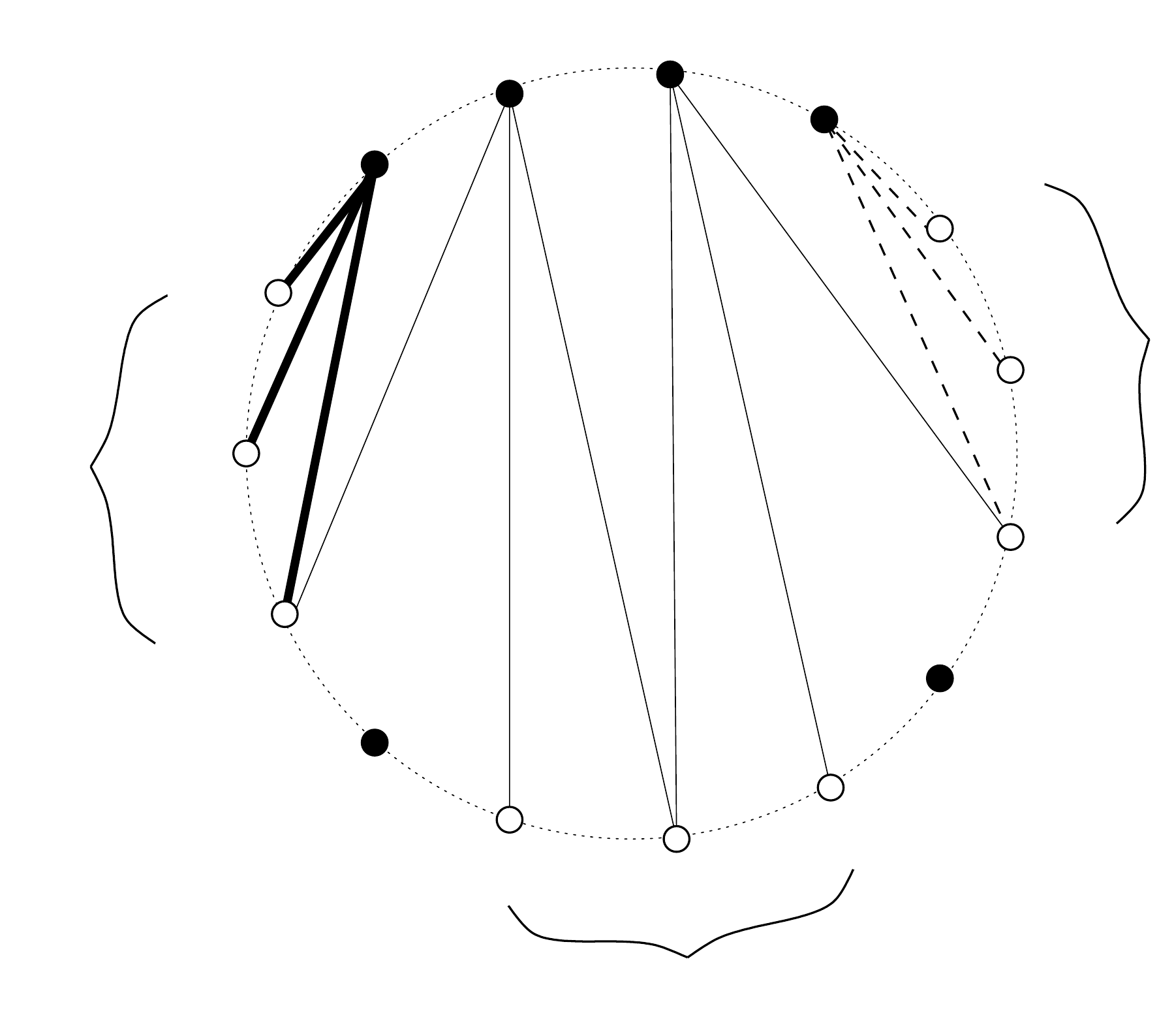}
\hglue 0.4 cm \scalebox{0.36}{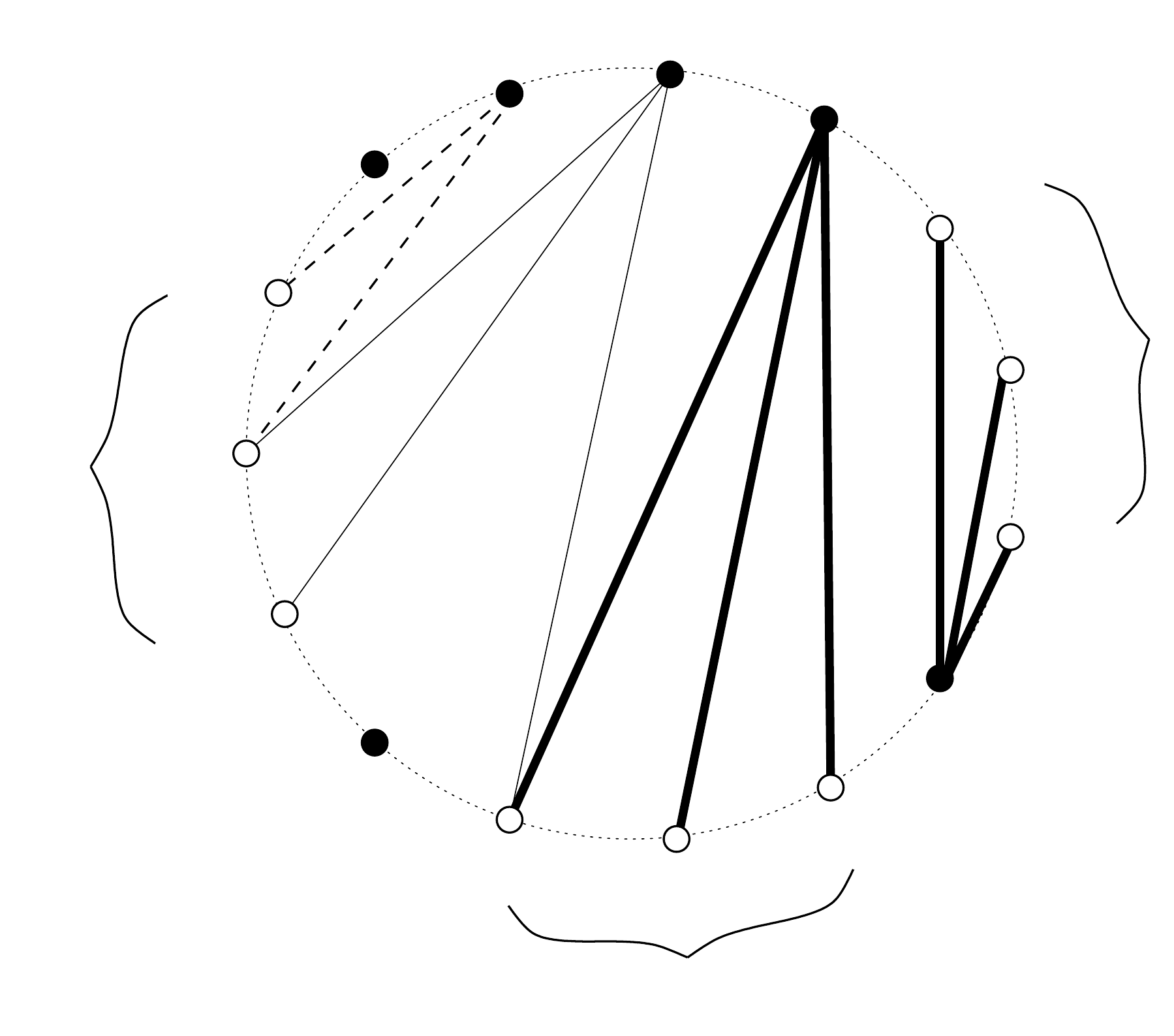}
\vglue -0.5 cm
\vglue -0.5 cm
\hglue -1.2 cm \begin{center}\scalebox{0.36}{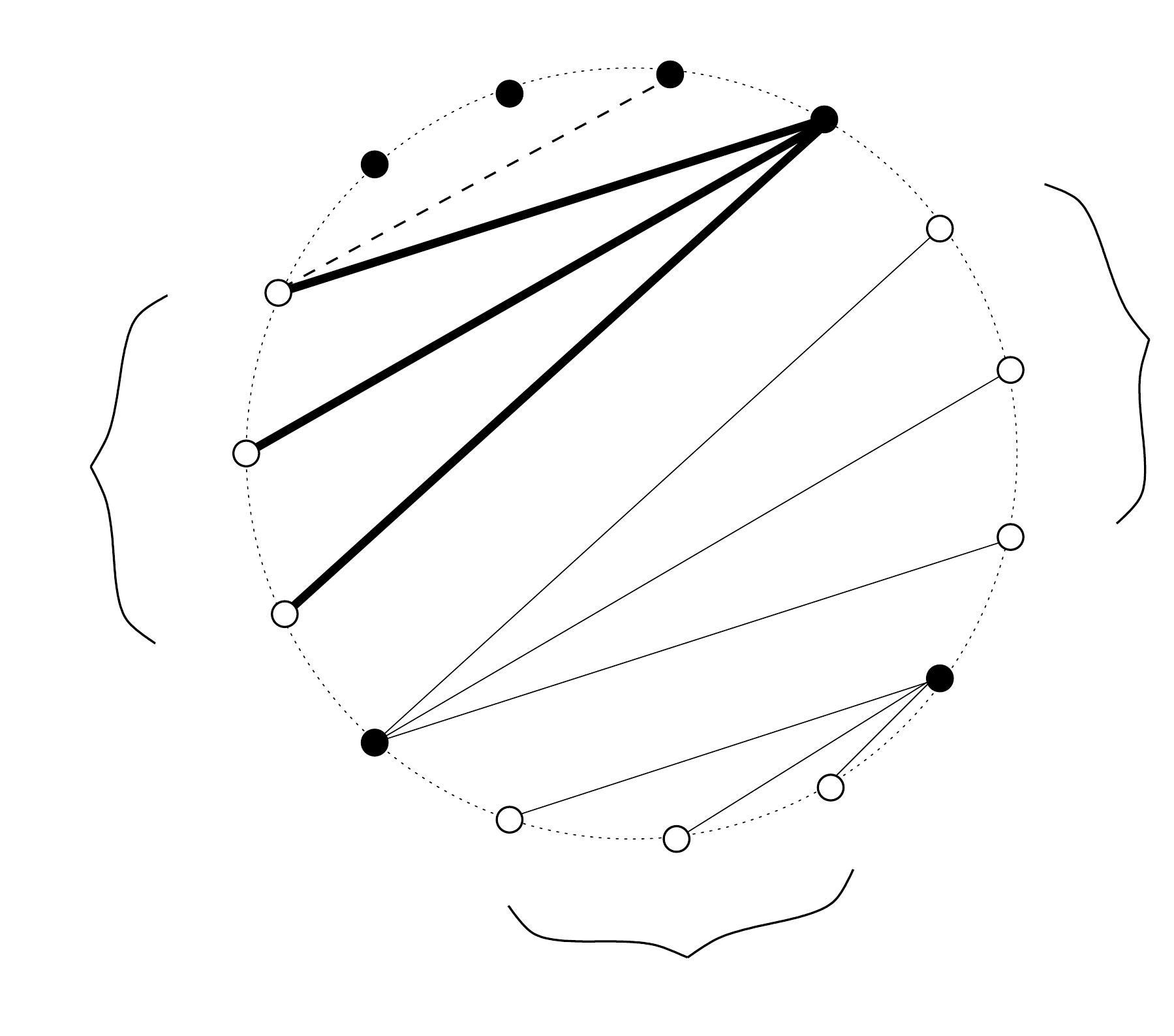}\end{center}
\caption{\small A balanced $5$-page embedding of $K_{6,9}$. In this
  case $k=5$, and so $s=t=3$. Pages $0,1,2,3$ and $4$ are the upper left, upper right,
  middle left, middle right, and lower circle, respectively.
For Pages $0, 1$, and $2$, we
  have edges of Types I, II, and III, whereas for Pages $3$ and $4$,
  we have edges of Types IV, V, and VI. Edges of Types I and IV are
  drawn with thick segments; edges of Types II and V are drawn with
  thinner segments; and edges of Types III and VI are drawn with
  dashed segments.
  }
\label{fig:easex01}
\end{figure}

Now we proceed to place the edges on the pages. We refer the reader to
Figures~\ref{fig:easex01} and~\ref{fig:easex02} for illustrations of the edges distributions
for the cases $k=5$ and $6$.
We remark that:
(i) operations on page numbers are modulo $s+t-1$;
(ii) operations on block indices are modulo $t$;
(iii) operations on the indices of black
  vertices are modulo $s+t$; and
(iv) operations on the indices of white vertices are modulo $st$.


For $r=0,1,\ldots,s-1$, place the following edges in page $r$:

\begin{description}

\item{{\sc Type I \hglue 0.4 cm}} For $i = r+1, r+2, \ldots, t$, the edges joining $b_{s+i}$ to
  all the vertices in the white block $W_{t+r-i}$ (note that $b_{s+t}=b_0$).

\item{{\sc Type II \hglue 0.25 cm}} For $0 < i < r+1$, the edges joining $b_i$ to all the vertices
  in $W[ rs-i(s-1):rs-(i-1)(s-1)]$.

\item{{\sc Type III \hglue 0.1 cm}} The edges joining $b_{r+1}$ to all the vertices in $W[ 0 : r]$.

\end{description}

For $r=s,s+1,\ldots,s+t-2$, place the following edges in page $r$:

\begin{description}

\item{{\sc Type IV \hglue 0.1 cm}} For $i=0,1,\ldots,r-s+1$, $b_{s+i}$ to all the
  vertices in the white block $W_{r-s-i+1}$.

\item{{\sc Type V \hglue 0.25 cm}} For $0 < i < s-r+t-1$, the edges joining $b_{s-i}$ to all the
  vertices in $W[(i+r-s+1)s-i:(i+r-s+1)s-i+(s-1)]$.

\item{{\sc Type VI \hglue 0.1 cm}} The edges joining $b_{r-t+1}$ to
  all the vertices in $W[st-t+r-s+1:st-1]$.

\end{description}

\begin{figure}[h!]
\begin{center}
\scalebox{0.38}{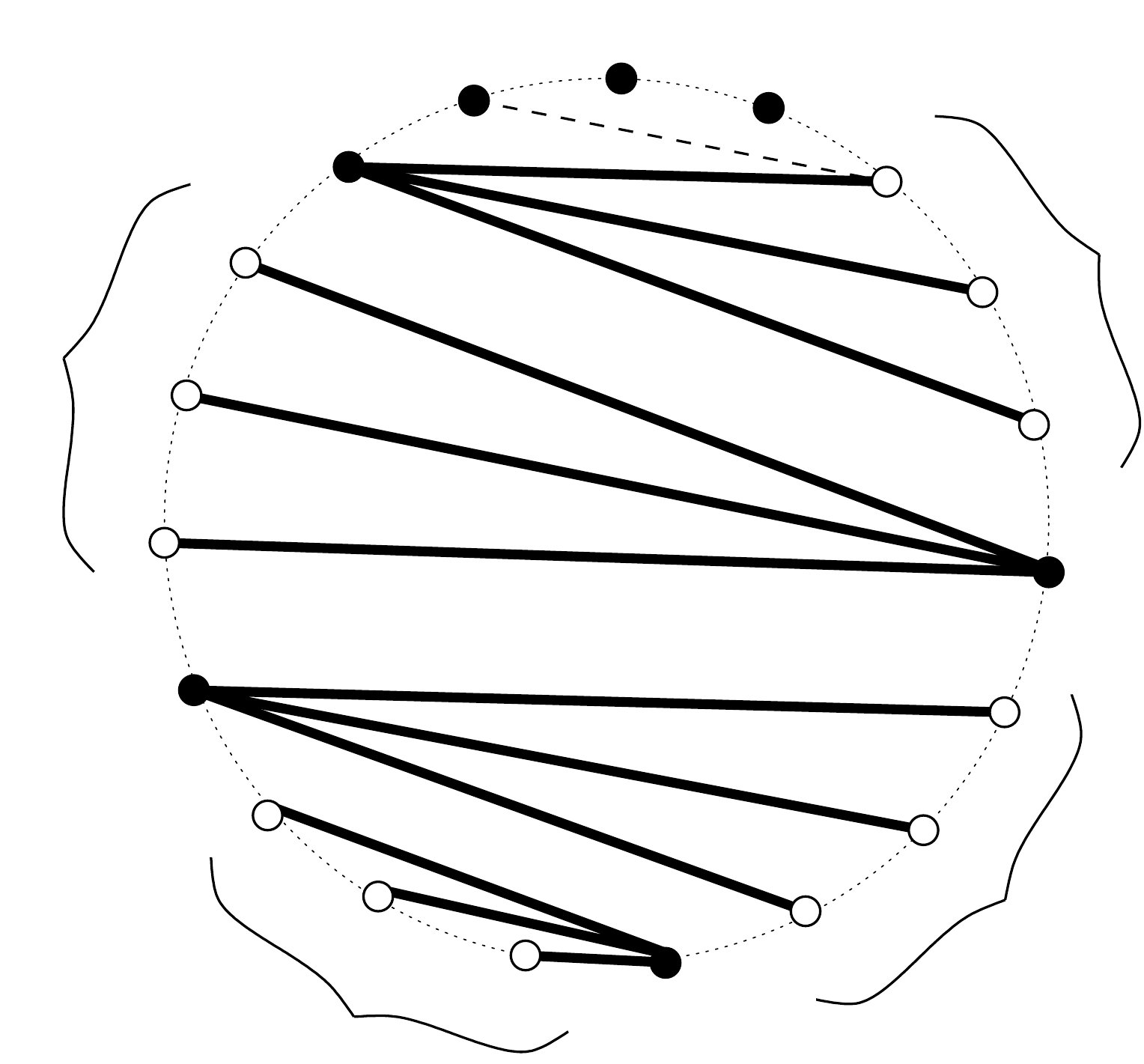}
\hglue 0.3 cm \scalebox{0.38}{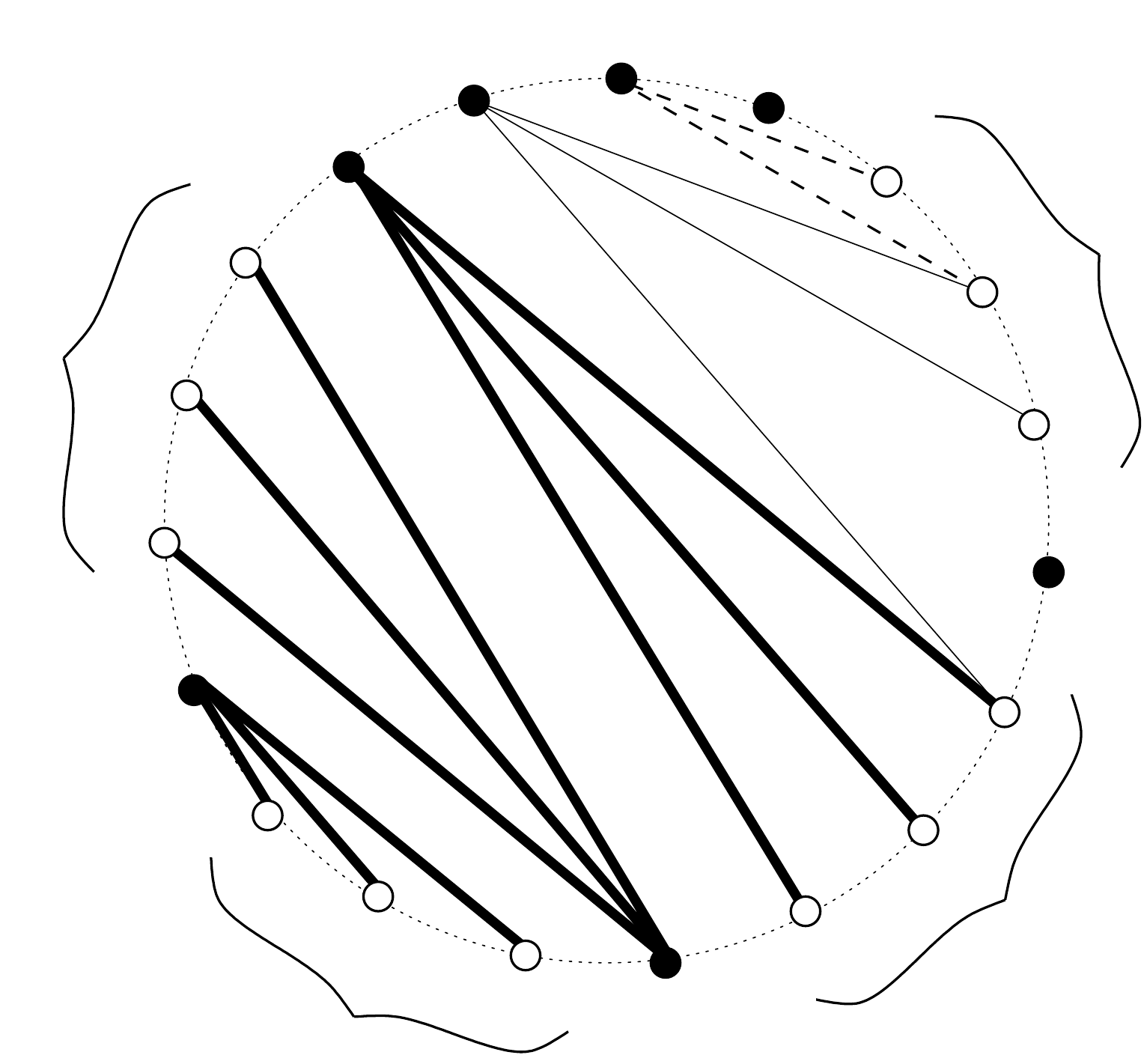}
\vglue 0.3 cm
\scalebox{0.38}{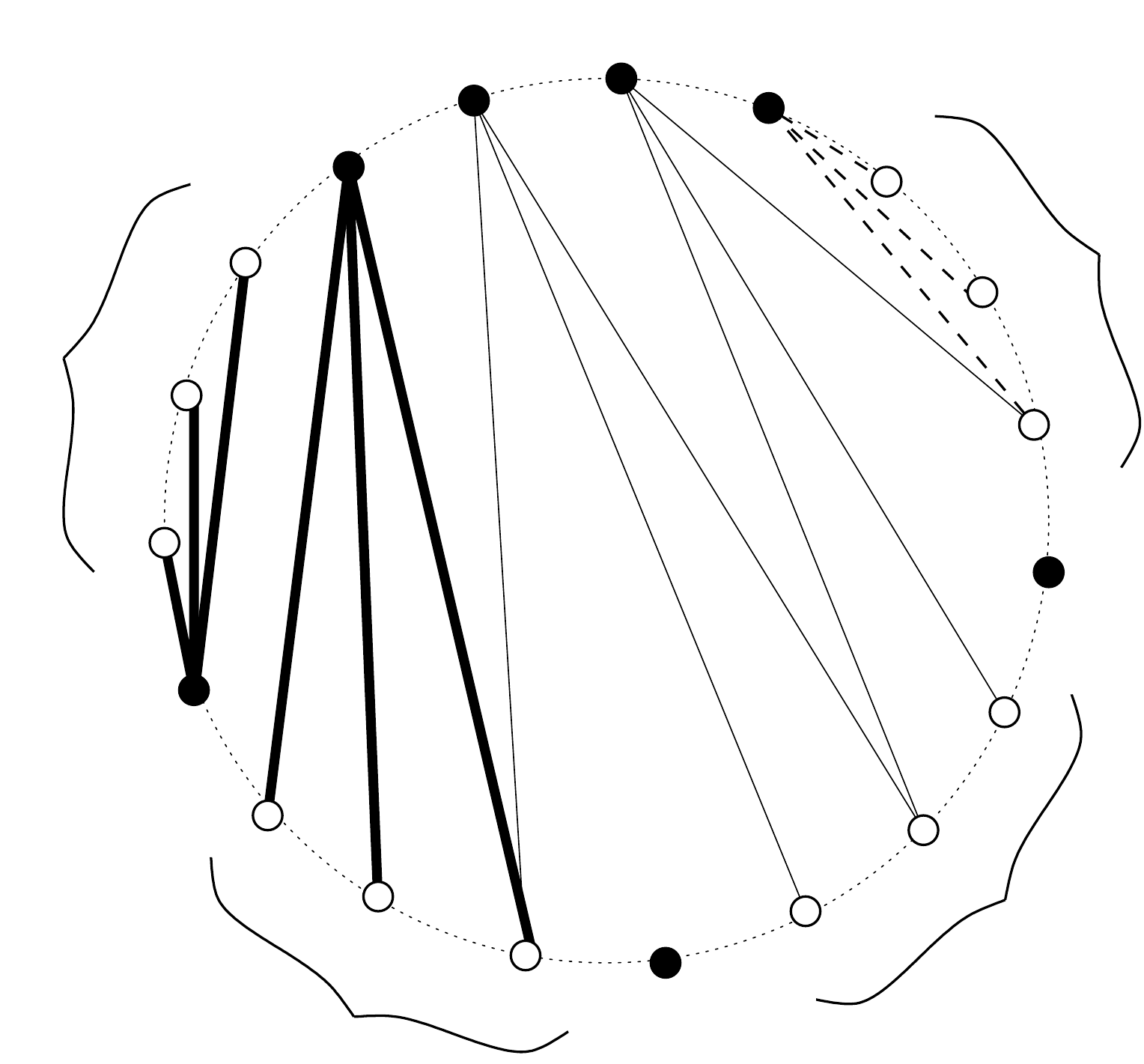}
\hglue 0.3 cm \scalebox{0.38}{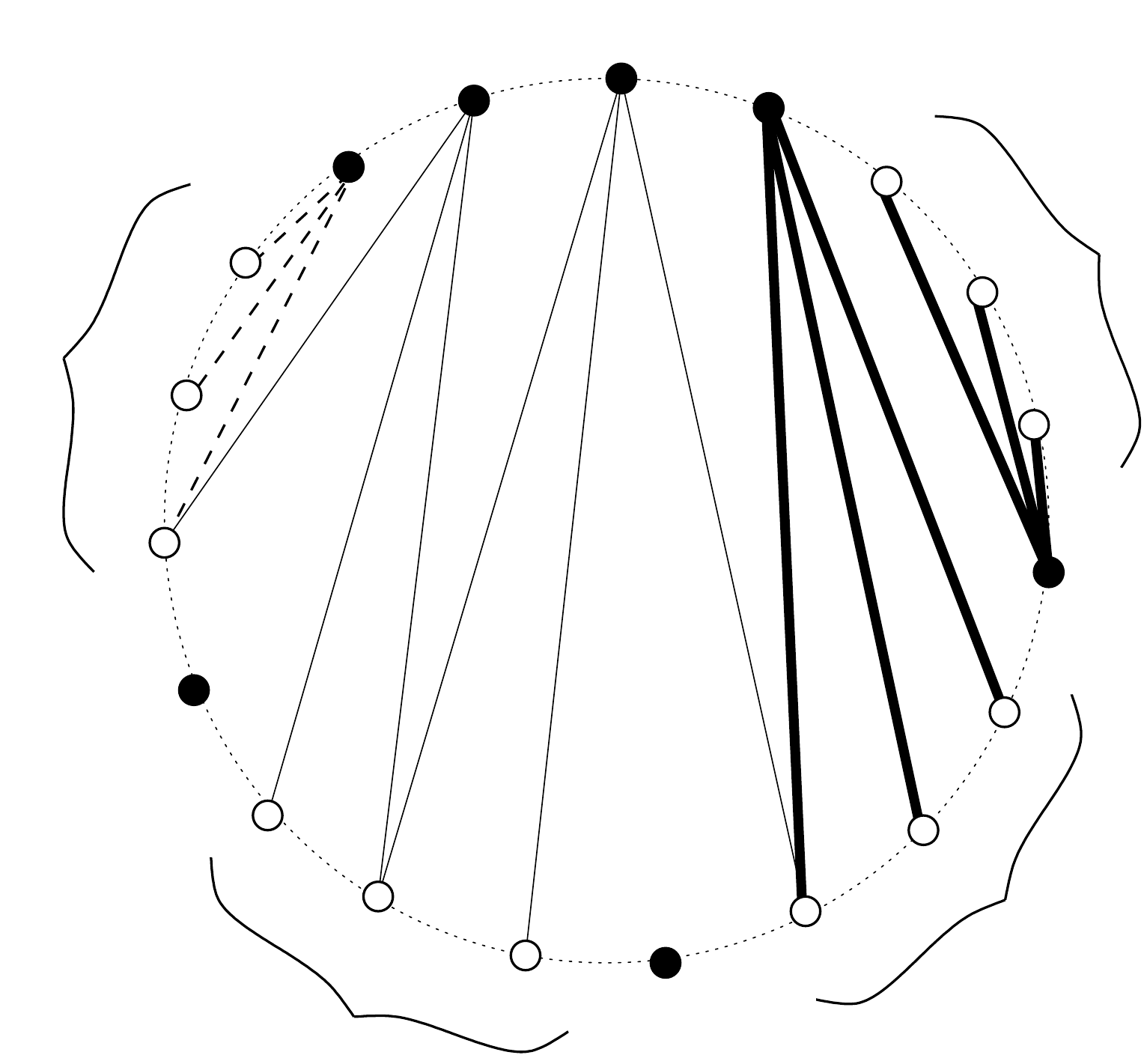}
\vglue 0.3 cm
\scalebox{0.38}{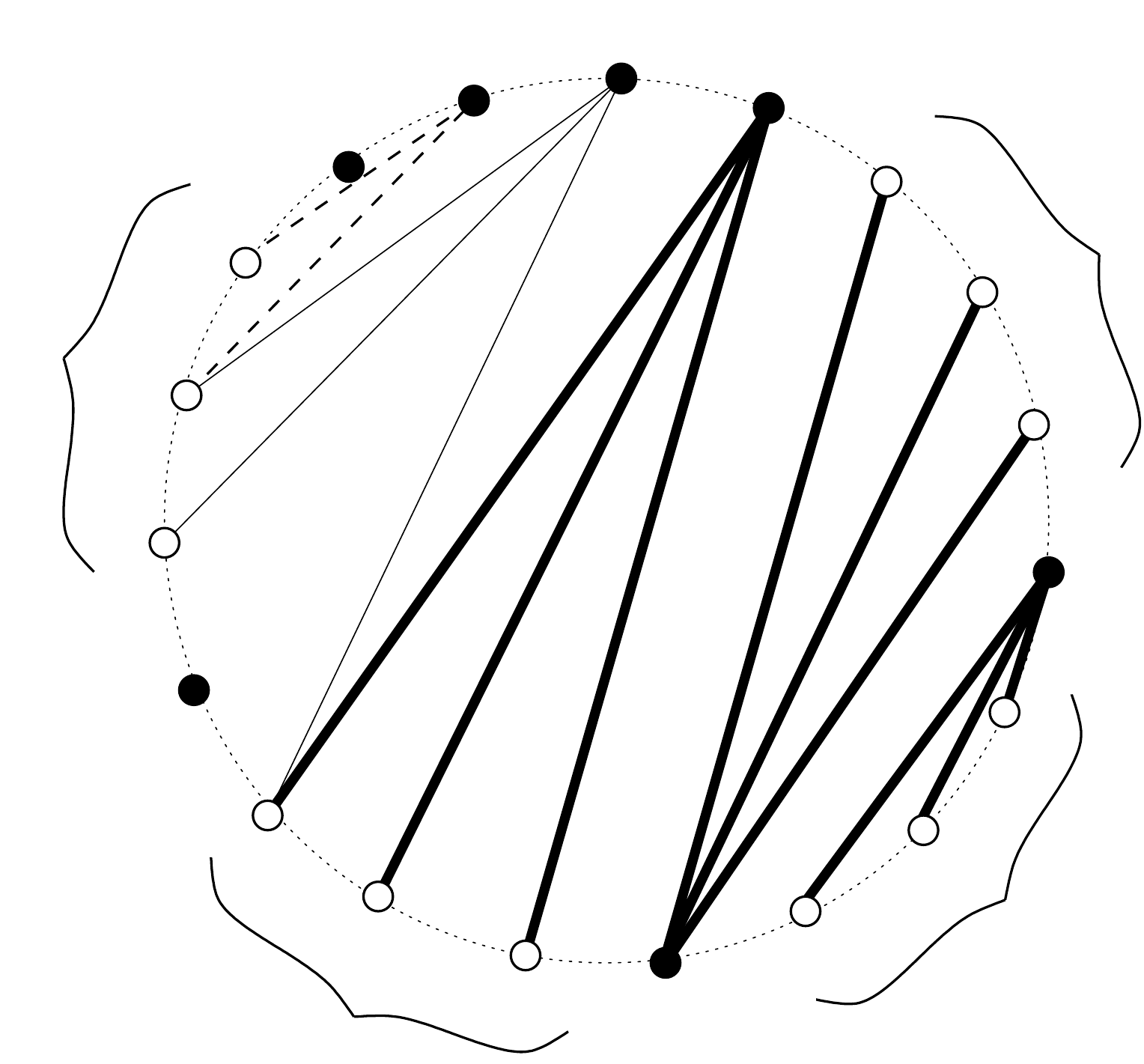}
\hglue 0.3 cm \scalebox{0.38}{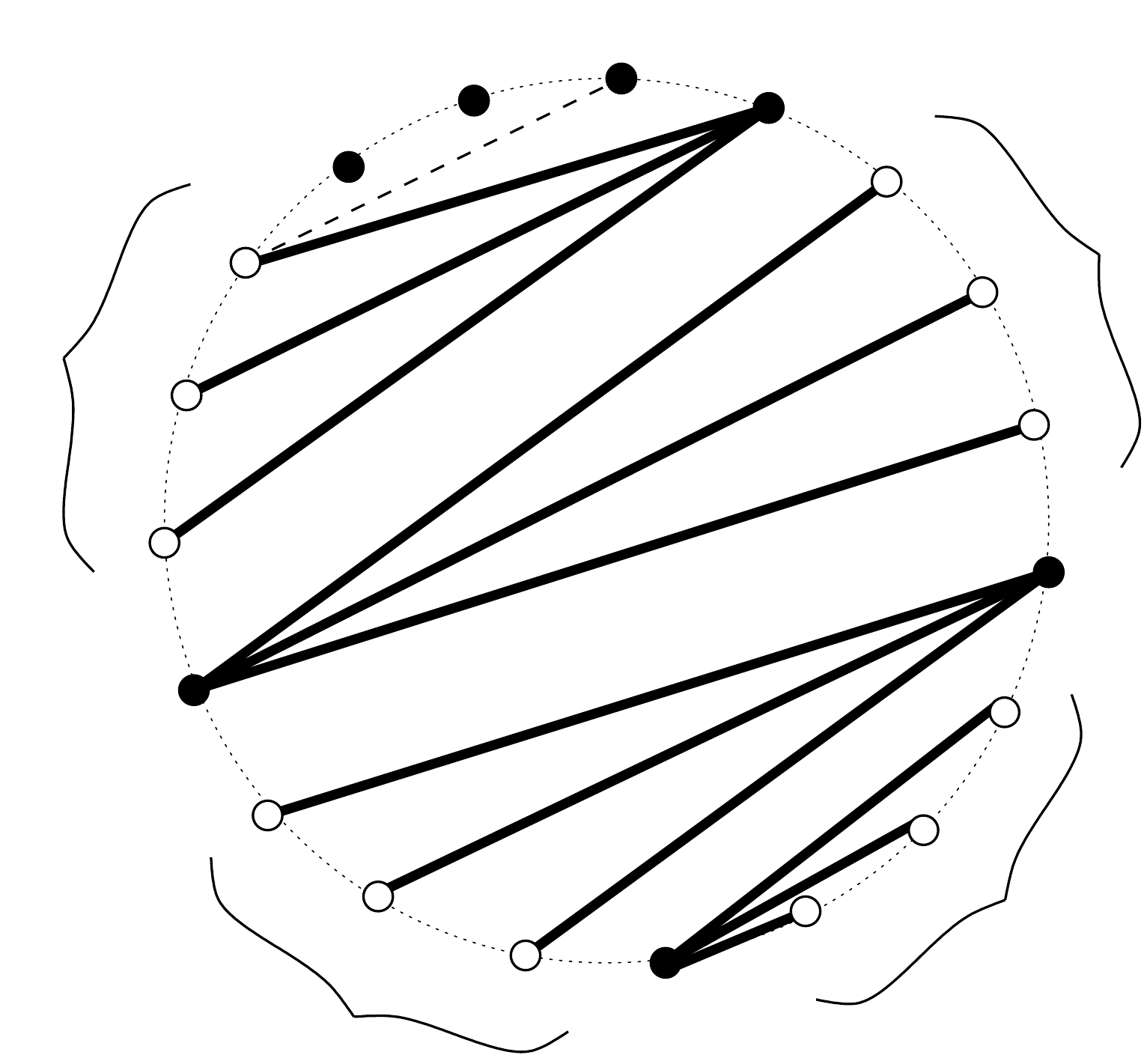}
\end{center}
\caption{\small A balanced $6$-page embedding of $K_{7,12}$. In this
  case $k=6$, and so $s=3$ and $t=4$. Pages $0,1,2,3,4,$ and $5$ are the upper left, upper right,
  middle left, middle right, lower left, and lower right circles, respectively.
For Pages $0, 1$, and $2$, we
have edges of Types I, II, and III, whereas for Pages $3,4$, and $5$,
  we have edges of Types IV, V, and VI. Edges of Types I and IV are
  drawn with thick segments; edges of Types II and V are drawn with
  thinner segments; and edges of Types III and VI are drawn with
  dashed segments.
  }
\label{fig:easex02}
\end{figure}

It is a tedious but  straightforward task to check that this yields an $(s+t-1)$-page
embedding of $K_{s+t,st}$. Moreover, since every white vertex has load
at least $1$ in every page, it follows immediately that the embedding
is balanced.
\end{proof}

\subsection{The upper bound}

\begin{lemma}\label{lem:corone}
For all positive integers $k$ and $n$,
\[
\nu_k(K_{k+1,n})
\le
q \cdot \binom{\frac{n-q}{\ell}+1}{2}
+
(\ell-q)\cdot \binom{\frac{n-q}{\ell}}{2}
,
\]
where $\ell:=\floor{(k+1)^2/4}$ and  $q:= \nmodk{n}{\floor{(k+1)^2/4}}$.
\end{lemma}

\begin{proof}
It follows immediately by combining Propositions~\ref{pro:upp1} and~\ref{pro:ek}.
\end{proof}

\section{Proofs of Theorems~\ref{thm:main1} and~\ref{thm:main2}}\label{sec:proofs}

We first observe that Theorem~\ref{thm:main1} follows immediately by
combining Lemmas~\ref{lem:exon} and~\ref{lem:corone}.

Now to prove Theorem~\ref{thm:main2}, we let $\ell:=\floor{(k+1)^2/4}$ and $q:=
\nmodk{n}{\floor{(k+1)^2/4}}$,
and note that it follows from
Lemma~\ref{lem:corone} that
\begin{align*}
\nu_k(K_{k+1,n})
&\le
q \cdot \binom{\frac{n-q}{\ell}+1}{2}
+
(\ell-q)\cdot \binom{\frac{n-q}{\ell}}{2}
\le
\ell\cdot \binom{\frac{n-q}{\ell}+1}{2}
=
\frac{\ell}{2}\cdot \biggl(\frac{n-q}{\ell} + 1 \biggr)
\biggl(\frac{n-q}{\ell}\biggr) \\
&= \frac{n-q}{2}\cdot \biggl(  \frac{n-q}{\ell} + 1
\biggr)
\le \frac{n}{2} \cdot \biggl( \frac{n}{\ell}+1 \biggr)
=
\frac{n^2}{2\ell} + \frac{n}{2}
\le
\frac{n^2}{2(k^2/4)} + \frac{n}{2} = \frac{2n^2}{k^2} + \frac{n}{2}.
\end{align*}
Combining this with Lemma~\ref{lem:exth}, we obtain
\[
2n^2
\biggl(
\frac{1}{k^2 + 2000k^{7/4}}
\biggr)
-n
<  \nu_k(K_{k+1,n}) \le
\frac{2n^2}{k^2} + \frac{n}{2},
\]
proving Theorem~\ref{thm:main2}.

\section{A general upper bound for $\nu_k(K_{m,n})$: proof of Theorem~\ref{thm:upp1}}\label{sec:cons}

We now describe a quite natural construction to draw $K_{m,n}$ in $k$
pages, for every $k\ge 3$. Actually, our construction also works for
the case $k=2$, and for this case the upper bounds obtained coincide with
the best known upper bound for $\nu_2(K_{m,n})$.

\begin{proof}[Proof of Theorem~\ref{thm:upp1}] 
For simplicity, we color the $m$ vertices black, and the $n$ vertices
white. Let $p,q,r,s$ be the nonnegative integers defined by the
conditions $m=kp+r$ and $0 \le r \le k-1$, and $n=kq+s$ and $0 \le s
\le k-1$ (note that the definitions of $r$ and $s$ coincide with those
in the statement of Theorem~\ref{thm:upp1}). Our task is to describe a drawing of $K_{m,n}$ with exactly ${(m-r)(n-s)}(m-k+r)(n-k+s)/(4k^2)$ crossings.

We start our construction by dividing the set of black vertices into
$k$ groups $B_0, B_1, \ldots, B_{k-1}$, so that ${k-r}$ of them
(say the first $k-r$) have size $p$, and the remaining $r$  have size
$p+1$. Then
we divide the set of white vertices into
$k$ groups $W_0, W_1, \ldots, W_{k-1}$, such that ${k-s}$ of them
(say the first $k-s$) have size $q$, and the remaining $s$ have size
$q+1$.

Then (using the circular drawing
model) we place the groups alternately on a circumference, as in $B_0,
W_0, B_1, W_1, \ldots, B_{k-1}, W_{k-1}$. Now for
$i=0,1,2,\ldots,k-1$,
we draw in page $i$ the edges joining all black points in $B_j$ to all white points
in $W_s$ if and only if $j+s = i$ (operations are modulo $k$).

A straightforward calculation shows that the total number of crossings
in this drawing is $(k-r)(k-s){p\choose 2}{q\choose 2}$
$+(k-r)s{p\choose 2}{q+1\choose 2}$
$+r(k-s){p+1\choose 2}{q\choose 2}$
$+rs{p+1\choose 2}{q+1\choose 2}$, and an  elementary manipulation
shows that this equals
${(m-r)(n-s)}(m-k+r)(n-k+s)/(4k^2)$.
Thus $\nu_k(K_{m,n}) \le
{(m-r)(n-s)}(m-k+r)(n-k+s)/(4k^2)$, as claimed.

Finally, note that since obviously $m-r \le m$, $n-s \le n$,
$m-k+r \le m-1$, and $n-k+s \le n-1$, it follows that
$\nu_k(K_{m,n}) \le (1/4k^2)m(m-1)n(n-1)=(1/k^2)\binom{m}{2}\binom{n}{2}$.
\end{proof}

\section{Concluding remarks}\label{sec:conrem}

It seems worth gathering in a single expression the best lower and
upper bounds we now have for $\nu_k(K_{m,n})$. Since $\nu_k(K_{m,n})$
may exhibit an exceptional behaviour for small values of $m$ and $n$,
it makes sense to express the asymptotic forms of these bounds. The
lower bound (coming from \cite[Theorem 5]{sssv07}) is given in
\eqref{eq:lobokp},
whereas the upper bound is from Theorem~\ref{thm:upp1}.
\begin{equation}\label{eq:summar}
\frac{1}{3(3\ceil{\frac{k}{2}}-1)^2}
\le
\lim_{m,n\to\infty} \frac{\nu_k(K_{m,n})}{\binom{m}{2}\binom{n}{2}}
\le
\frac{1}{k^2}.
\end{equation}

\def\kpagebip{{\text{\sc BookBipartite}}}
\def\kpagecom{{\text{\sc BookComplete}}}
\def\multicom{{\text{\sc MultiplanarComplete}}}
\def\multibip{{\text{\sc MultiplanarBipartite}}}

As we have observed (and used) above, $\ucr_{k/2}(K_{m,n}) \le
\nu_k(K_{m,n})$, and it is natural to ask whether $\nu_k(K_{m,n})$
is strictly greater than $\ucr_{k/2}(K_{m,n})$ (we assume $k$ even in
this discussion). At least in principle, there
is much more freedom in $k/2$-planar drawings than in $k$-page
drawings. Thus
remains the question: can this
additional freedom be used to (substantially) save crossings?

With this last question in mind, we now carry over an exercise
which reveals the connections
between the book and multiplanar crossing numbers of
complete and complete bipartite graphs.

It is not difficult to prove that the constants
\begin{align*}
\kpagebip \, &:= \,
\lim_{k\to\infty} \, k^2
\cdot \biggl(\lim_{m,n\to\infty} \frac{\nu_k(K_{m,n})}{\binom{m}{2}\binom{n}{2}}\biggr),
\\
\kpagecom \, &:= \, \lim_{k\to\infty} \, k^2
\cdot \biggl(\lim_{n\to\infty} \frac{\nu_k(K_{n})}{\binom{n}{4}}\biggr),
\\
\multibip \, &:= \, \lim_{k\to\infty} \, k^2
\cdot \biggl(\lim_{m,n\to\infty} \frac{\ucr_k(K_{m,n})}{\binom{m}{2}\binom{n}{2}}\biggr),
\\
\multicom \, &:= \, \lim_{k\to\infty} \, k^2
\cdot \biggl(\lim_{n\to\infty} \frac{\ucr_k(K_{n})}{\binom{n}{4}}\biggr),
\end{align*}
are all well-defined.

In view of \eqref{eq:summar}, we have
\begin{equation}\label{eq:bip3}
\frac{4}{27}
\le
\kpagebip
\le
1.
\end{equation}

Using the best known upper bound for $\ucr_{k}(K_{m,n})$
(from~\cite[Theorem 8]{sssv07}), we obtain
\begin{equation}\label{eq:suppo}
\multibip \le \frac{1}{4}.
\end{equation}

We also invoke the upper bound
$\ucr_k(K_n) \le (1/64)k(n+k^2)^4/(k-1)^3$, which holds whenever $k$
is a power of a prime and $n \ge (k-1)^2$ (see~\cite[Theorem
  7]{sssv07}). This immediately yields
\begin{equation}\label{eq:aaa}
\multicom \le \frac{3}{8}.
\end{equation}

We also note that the observation $\ucr_{k/2}(K_{m,n}) \le \nu_k(K_{m,n})$
immediately implies that
\begin{equation}\label{eq:someq}
\multibip \le \frac{1}{4}\cdot \kpagebip.
\end{equation}

Finally,
applying the Richter-Thomassen counting argument~\cite{rt} for
bounding the crossing number of $K_{2n}$ in terms of the crossing number
of $K_{n,n}$ (their argument applies unmodified to $k$-planar crossing
numbers), we obtain
\begin{equation}\label{eq:whaa}
\multicom \ge \frac{3}{2} \cdot \multibip.
\end{equation}

Suppose that the multiplanar drawings of Shahrokhi et
al.~\cite{sssv07} are asymptoticall optimal. In other words, suppose
that equality holds in \eqref{eq:suppo}. Using \eqref{eq:someq}, we
obtain $\kpagebip \ge 1$, and by \eqref{eq:bip3} then we get
$\kpagebip = 1$. Moreover (again, assuming equality holds in
\eqref{eq:suppo}), using \eqref{eq:whaa}, we obtain $\multicom \ge
3/8$, and so in view of \eqref{eq:aaa} we get $\multicom = 3/8$.
Summarizing:

\begin{observation}
Suppose that the multiplanar drawings of $K_{m,n}$ of Shahrokhi et
al.~\cite{sssv07} are asymptotically optimal, so that
$\multibip = \frac{1}{4}$. Then
\begin{align*}
\kpagebip &= 1,   \text{\hglue 0.3 cm {\rm and }}\\
\text{\hglue 4.8 cm }\multicom &= \frac{3}{8}. \text{\hglue 4.8 cm } \Box
\end{align*}
\end{observation}

In other words, under this scenario (the multiplanar drawings of
$K_{m,n}$ in~\cite{sssv07} being asymptotically optimal), the additional freedom
of $(k/2)$-planar over $k$-page drawings of $K_{m,n}$ becomes less and less important as
the number of planes and pages grows. In addition, under this scenario
the $k$-planar crossing number of $K_n$ also gets (asymptotically) determined.

\paragraph{Acknowledgements.}
The authors are grateful to Cesar Hernandez-Velez and to Imrich Vrt'o
for helpful comments.


\begin{thebibliography}{99}

\bib{beineke}{article}{
   author={Beineke, Lowell W.},
   title={Complete bipartite graphs: Decomposition into planar subgraphs},
   conference={
      title={A seminar on Graph Theory},
   },
   book={
      publisher={Holt},
      place={Rinehart and Winston, New York},
   },
   date={1967},
   pages={42--53},
   review={\MR{0215745 (35 \#6580)}},
}
\bib{Bernhart-Kainen}{article}{
   author={Bernhart, Frank},
   author={Kainen, Paul C.},
   title={The book thickness of a graph},
   journal={J. Combin. Theory Ser. B},
   volume={27},
   date={1979},
   number={3},
   pages={320--331},
   issn={0095-8956},
}

\bib{dsdp}{article}{
  author =       {Benson, Steven J.},
  author =       {Ye, Yinyu},
  author =       {Zhang, Xiong},
  title =        {Solving Large-Scale Sparse Semidefinite Programs for
                 Combinatorial Optimization},
  year =         {2000},
  journal =      {SIAM Journal on Optimization},
  volume =       {10},
  number =        {2},
  pages =        {443--461},
}

\bib{Buccheim-Zheng}{article}{
   author = {Buchheim, Christoph},
   author ={Zheng, Lanbo},
   title = {Fixed Linear Crossing Minimization by Reduction to the Maximum Cut Problem},
   booktitle = {Computing and Combinatorics},
   series = {Lecture Notes in Computer Science},
   editor = {Chen, Danny and Lee, D.},
   publisher = {Springer Berlin / Heidelberg},
   isbn = {978-3-540-36925-7},
   pages = {507-516},
   volume = {4112},
   year = {2006}
}



\bib{leighton}{article}{
   author={Chung, Fan R. K.},
   author={Leighton, Frank Thomson},
   author={Rosenberg, Arnold L.},
   title={Embedding graphs in books: a layout problem with applications to
   VLSI design},
   journal={SIAM J. Algebraic Discrete Methods},
   volume={8},
   date={1987},
   number={1},
   pages={33--58},
   issn={0196-5212},
}
		

\bib{bip1}{article}{
   author={Czabarka, {\'E}.},
   author={S{\'y}kora, O.},
   author={Sz{\'e}kely, L. A.},
   author={Vr{{t}}'o, I.},
   title={Biplanar crossing numbers. I. A survey of results and problems},
   conference={
      title={More sets, graphs and numbers},
   },
   book={
      series={Bolyai Soc. Math. Stud.},
      volume={15},
      publisher={Springer},
      place={Berlin},
   },
   date={2006},
   pages={57--77},
}
		
\bib{bip2}{article}{
   author={Czabarka, {\'E}va},
   author={S{\'y}kora, Ondrej},
   author={Sz{\'e}kely, L{\'a}szl{\'o} A.},
   author={Vr{{t}}'o, Imrich},
   title={Biplanar crossing numbers. II. Comparing crossing numbers and
   biplanar crossing numbers using the probabilistic method},
   journal={Random Structures Algorithms},
   volume={33},
   date={2008},
   number={4},
   pages={480--496},
   issn={1042-9832},
}

\bib{dp}{article}{
   author={de Klerk, Etienne},
   author={Pasechnik, Dmitrii V.},
   title={Improved lower bounds for the 2-page crossing numbers of $K_{m,n}$
  and $K_n$ via semidefinite programming},
   journal={ SIAM J.~Opt.},
   volume={22},
   number={2},
   pages={581--595},
}

\bib{dps}{article}{
   author={de Klerk, Etienne},
   author={Pasechnik, Dmitrii V.},
   author={Schrijver, Alexander},
   title={Reduction of symmetric semidefinite programs using the regular
   $\ast$-representation},
   journal={Math. Program.},
   volume={109},
   date={2007},
   number={2-3, Ser. B},
   pages={613--624},
   issn={0025-5610},
}

\bib{MR}{article}{
  author={de Klerk, Etienne},
  author={Pasechnik, Dmitrii V.},
  author={Salazar, Gelasio},
  title={Improved lower bounds on book crossing numbers of complete graphs},
  status={Manuscript (2012)},
  eprint={arXiv:1207.5701 [math.CO]},
}

\bib{dw}{article}{
   author={Dujmovi{\'c}, Vida},
   author={Wood, David R.},
   title={On linear layouts of graphs},
   journal={Discrete Math. Theor. Comput. Sci.},
   volume={6},
   date={2004},
   number={2},
   pages={339--357},
   issn={1365-8050},
}

\bib{Enomoto}{article}{
   author={Enomoto, Hikoe},
   author={Nakamigawa, Tomoki},
   author={Ota, Katsuhiro},
   title={On the pagenumber of complete bipartite graphs},
   journal={J. Combin. Theory Ser. B},
   volume={71},
   date={1997},
   number={1},
   pages={111--120},
   issn={0095-8956},
}




\bib{Kainen}{article}{
   author={Kainen, Paul C.},
   title={Some recent results in topological graph theory},
   conference={
      title={Graphs and combinatorics (Proc. Capital Conf., George
      Washington Univ., Washington, D.C., 1973)},
   },
   book={
      publisher={Springer},
      place={Berlin},
   },
   date={1974},
   pages={76--108. Lecture Notes in Math., Vol. 406},
}
		

\bib{kleitman}{article}{
   author={Kleitman, Daniel J.},
   title={The crossing number of $K_{5,n}$},
   journal={J. Combinatorial Theory},
   volume={9},
   date={1970},
   pages={315--323},
}



\bib{akmaxsat}{article}{
  author={A.~K\"ugel},
  title={Improved Exact Solver for the Weighted Max-SAT Problem},
  conference={
    title={Proceedings of POS-10. Pragmatics of SAT},
    date={2012}
    },
book={
series={EasyChair Proceedings in Computing},
volume={8},
date={2012},
},
  pages={15--27}
  }


\bib{xpressmp}{article}{
 author = {Laundy, Richard},
 author = {Perregaard, Michael},
 author = {Tavares, Gabriel},
 author = {Tipi, Horia},
 author = {Vazacopoulos, Alkis},
 title = {Solving Hard Mixed-Integer Programming Problems with Xpress-MP: A MIPLIB 2003 Case Study},
 journal = {INFORMS J. on Computing},
 volume = {21},
 number = {2},
 year = {2009},
 issn = {1526-5528},
 pages = {304--313},
 publisher = {INFORMS},
 address = {Institute for Operations Research and the Management Sciences (INFORMS), Linthicum, Maryland, USA},
}

\bib{Lovasztheta}{article}{		
  author = {Lov{\'a}sz, L.},
  title = {On the {S}hannon capacity of a graph},
  journal = {IEEE Trans. on Information Theory},
  volume = {25},
  pages = {1--7},
  year = {1979}
}

\bib{muder}{article}{
   author={Muder, Douglas J.},
   author={Weaver, Margaret Lefevre},
   author={West, Douglas B.},
   title={Pagenumber of complete bipartite graphs},
   journal={J. Graph Theory},
   volume={12},
   date={1988},
   number={4},
   pages={469--489},
   issn={0364-9024},
}

\bib{rt}{article}{
   author={Richter, R. Bruce},
   author={Thomassen, Carsten},
   title={Relations between crossing numbers of complete and complete
   bipartite graphs},
   journal={Amer. Math. Monthly},
   volume={104},
   date={1997},
   number={2},
   pages={131--137},
   issn={0002-9890},
}

\bib{riskin}{article}{
   author={Riskin, Adrian},
   title={On the outerplanar crossing numbers of $K_{m,n}$},
   journal={Bull. Inst. Combin. Appl.},
   volume={39},
   date={2003},
   pages={16--20},
   issn={1183-1278},
}
		

\bib{sssv96}{article}{
   author={Shahrokhi, Farhad},
   author={Sz{\'e}kely, L{\'a}szl{\'o} A.},
   author={S{\'y}kora, Ondrej},
   author={Vr{\soft{t}}o, Imrich},
   title={The book crossing number of a graph},
   journal={J. Graph Theory},
   volume={21},
   date={1996},
   number={4},
   pages={413--424},
   issn={0364-9024},
}

		

\bib{sssv07}{article}{
   author={Shahrokhi, Farhad},
   author={S{\'y}kora, Ondrej},
   author={Sz{\'e}kely, L{\'a}szl{\'o} A.},
   author={Vr{\soft{t}}o, Imrich},
   title={On $k$-planar crossing numbers},
   journal={Discrete Appl. Math.},
   volume={155},
   date={2007},
   number={9},
   pages={1106--1115},
   issn={0166-218X},
}



\bib{woodall}{article}{
   author={Woodall, D. R.},
   title={Cyclic-order graphs and Zarankiewicz's crossing-number conjecture},
   journal={J. Graph Theory},
   volume={17},
   date={1993},
   number={6},
   pages={657--671},
   issn={0364-9024},
}


\bib{Yannanakis}{article}{
   author={Yannakakis, Mihalis},
   title={Embedding planar graphs in four pages},
   note={18th Annual ACM Symposium on Theory of Computing (Berkeley, CA,
   1986)},
   journal={J. Comput. System Sci.},
   volume={38},
   date={1989},
   number={1},
   pages={36--67},
   issn={0022-0000},
}
				
\bib{zaran}{article}{
   author={Zarankiewicz, K.},
   title={On a problem of P. Turan concerning graphs},
   journal={Fund. Math.},
   volume={41},
   date={1954},
   pages={137--145},
}


\end{thebibliography}
\end{document}